\documentclass[12pt]{article}

\usepackage{hyperref}

\usepackage{latexsym}
\usepackage{amsmath}
\usepackage{theorem}
\usepackage{amsmath}
\usepackage{wasysym} 
\usepackage{amssymb} 
\usepackage{epsfig,subfigure}
\usepackage{graphicx}

\usepackage{rotating}  
\usepackage{upgreek} 
\usepackage{txfonts}
\usepackage{pxfonts}
\usepackage{pstricks} 
\usepackage{multirow}
\usepackage{epsfig}
\usepackage{makecell}

\usepackage[normalem]{ulem}

\usepackage{tikz}

\usetikzlibrary{shapes.geometric}
\usetikzlibrary{positioning}

\usetikzlibrary{arrows}
\usetikzlibrary{automata}

\newcommand{\RR}{{\rm\bf R}}
\newcommand{\mm}{\mbox{$\mathfrak m$}}

\newtheorem{thm}{Theorem}[section]
\newtheorem{lem}[thm]{Lemma}
\newtheorem{cor}[thm]{Corollary}
\newtheorem{proposition}[thm]{Proposition}
\newtheorem{Def}[thm]{Definition}
\newtheorem{rmk}[thm]{Remark}

\newtheorem{ex}[thm]{Example}

\newenvironment{proof}
{\noindent\textbf{Proof}\hspace*{1em}\normalfont}

\newcommand{\qed}{\hfill $\Box$ 
}
\newcommand{\END}{\hfill\mbox{\raggedright$\Diamond$}}

\newcommand{\etal}{{\em et al.}}

{\begingroup

  \begin{enumerate}}
  {\end{enumerate}\endgroup}

\newcommand{\BB}{\mbox{$\mathcal B$}}
\newcommand{\CC}{\mbox{$\mathcal C$}}

\newcommand{\EE}{\mbox{$\mathcal E$}}
\newcommand{\E}{\mbox{{\footnotesize $\mathcal E$}}}

\newcommand{\MM}{\mbox{$\mathcal M$}}
\newcommand{\NN}{\mbox{$\mathcal N$}}
\newcommand{\N}{\mbox{{\footnotesize $\mathcal N$}}}
\newcommand{\LL}{\mbox{$\mathcal L$}}

\newcommand{\TT}{\mbox{$\mathcal T$}}
\newcommand{\T}{\mbox{{\footnotesize $\mathcal T$}}}

\newcommand{\candy}{\text{\scalebox{.9}{$\triangleright\hskip-.15cm\circ\hskip-.15cm\triangleleft\hskip0cm$}}}

\title{Networks with asymmetric inputs:  lattice of synchrony subspaces}

\author{
M.A.D. Aguiar$ $\thanks{Partial support by CMUP (UID/MAT/00144/2013), which is funded by FCT (Portugal) with national (MEC) and European structural funds through the programs FEDER, under the partnership agreement PT2020} \\
{\small Centro de Matem\'atica da Universidade do Porto,}\\
{\small Rua do Campo Alegre, 687, 4169-007 Porto, Portugal}\\
{\small Faculdade de Economia, Universidade do Porto,} \\
{\small Rua Dr Roberto Frias, 4200-464 Porto, Portugal}\\
\\
 {\small E-mail: maguiar@fep.up.pt}}

\date{\today}

\begin{document}

\maketitle 

\begin{abstract}

We consider coupled cell networks with asymmetric inputs  and study their lattice of synchrony subspaces. For the particular case of 1-input regular coupled cell networks we describe the join-irreducible synchrony subspaces for their lattice of synchrony subspaces, first in terms of the eigenvectors and generalized eigenvectors that generate them, and then by giving a characterization of the possible patterns of the associated balanced colourings. The set of the join-irreducible synchrony subspaces is join-dense for the lattice, that is, the lattice can be obtained by sums of those join-irreducible elements (M. Aguiar, P. Ashwin, A. Dias, and M. Field. 
Dynamics of coupled cell networks: synchrony, heteroclinic cycles and inflation,  
{\em J. Nonlinear Sci.} {\bf 21} (2) (2011) 271--323), and we conclude about the possible patterns of balanced colourings associated to the synchrony subspaces in the lattice.
We also consider the disjoint union of two regular coupled cell networks with the same cell-type and the same edge-type. We show how to obtain the lattice of synchrony subspaces for the network union from the lattice of synchrony subspaces for the component networks. The lattice of synchrony subspaces for a homogeneous coupled cell network is given by the intersection of the lattice of synchrony subspaces for its identical-edge subnetworks per each edge-type 
(M. A. D. Aguiar and A. P. S. Dias. The lattice of synchrony subspaces of a coupled cell network: Characterization and computation algorithm,
{\em Journal of Nonlinear Science}, {\bf 24} (6) (2014), 949--996). This, together with the results in this paper, on the lattice of synchrony subspaces for 1-input regular networks and on the lattice of synchrony subspaces for the disjoint union of networks, define a procedure to obtain the lattice of synchrony subspaces for homogeneous coupled cell networks with asymmetric inputs.

\vspace{3mm}


\vspace{0.1in}


\end{abstract}

\section{Introduction}

Coupled cell systems have been for long a focus of interest to the scientific community, including  biologists, physicists and mathematicians, since these systems are used as models in a wide range of real-world applications. See, for example, Albert and Barab\'asi~\cite{AB02}, Newman~\cite{N03}, Boccaletti~\etal~\cite{BLMCH06}, Arenas~\etal~\cite{ADKZ08}, and references therein. The structure of a coupled cell system can be abstracted by a {\em coupled cell network} - each cell represents an individual dynamical system and the connections represent the mutual interactions between those individual dynamics. See, for example, the formalism of Golubitsky and Stewart~\cite{SGP03}, \cite{GST05}, \cite{GS05} which is more algebraic and the formalism of Field~\cite{F04} which is more combinatorial. 

Coupled cell networks can be represented by directed graphs where the vertices are the cells and the edges represent the connections between them. Edges of the same type indicate the same kind of interaction, that is, the same coupling function.
There is extensive work dedicated to the study of statistical properties of complex networks, that is, networks with non-trivial topological features, which do not occur in simple networks such as lattices or random graphs. See, for example, the reviews Albert and Barab\'asi~\cite{AB02} and Newman~\cite{N03}, and references therein. A challenging problem with interest from the point of view of applications is how to infer, from measured real data, the connectivity structure, as well as the coupling functions associated to the different interactions, of a network. See, for example, the work of Stankovski~\etal~\cite{STMS15} and the very recent review Stankovski~\etal~\cite{SPMS17}. As noted in Newman~\cite{N03}, although there are excellent results on statistical properties that characterize the structure and behaviour of networked systems, and on the modelling of networks that can help to understand the meaning of those properties, there are still few studies about the effects of the network structure on the dynamical system behaviour. One of the main aims in the study of coupled cell systems is to characterize the dynamical properties of the systems that are admissible by a network based only on the network structure and independent of the specific dynamics at the nodes and the specific coupling functions. One such relevant dynamical property is the phenomena of synchronization. 

In the literature there are two main notions of {\em synchronization}: one, which is most often discussed in the context of phase oscillators, where the coupling structure may induce synchronization in the sense that all phases become equal at some point or points in time; the other, which we consider here, where a subset of the coupled identical dynamical systems, identical in the sense of having the  same state space and the same internal dynamics, follow exactly the same dynamics, that is, considering the same initial condition in the state space of each of those systems, they follow the same trajectory in the respective state space. This second concept is sometimes referred to as {\em cluster synchronization}. See Abrams~\etal~\cite{APM16} to an insight into the history of these two types of synchronization. In \cite{ADKZ08},  Arenas~\etal \ revise research undertaken to understand the impact of a variety of topological structures of interactions on synchronization properties. In Watts and Strogatz~\cite{WS98},  it is shown that models of dynamical systems with small-world coupling structure enhance synchronizability. For pattern formation and synchronization of coupled oscillators, see Kuramoto~\cite{K12}.

In this work we consider the patterns of synchronization that can occur for coupled cell networks with asymmetric inputs. Those patterns of synchronization correspond to the {\em synchrony subspaces} for the network - subspaces defined by equalities of cell coordinates and that are flow-invariant by all the coupled cell systems that are compatible with the network structure; in particular, systems with additive input structure, which are relevant from the point of view of applications as they are commonly used to model coupled oscillators, see Kuramoto~\cite{K12}.
The characterization of the set of synchrony subspaces  for a network is important since the existence of these flow-invariant subspaces can have a strong impact on the dynamics and favor the existence of non-generic dynamical behavior like robust heteroclinic cycles and networks and bifurcation phenomena. See, for example, \cite{AADF11}, \cite{AP13}, \cite{F14}, \cite{AD17}, \cite{F17} and \cite{GNS04}, \cite{DP10}, \cite{GL09}, \cite{SG11}.

Networks where all cells are identical, in the sense that they have the same phase space and internal dynamics, and such that the number of input edges per edge-type is the same for all cells are called {\em homogeneous}.  The structure of such networks can be described by adjacency matrices: for each edge-type there is one {\em adjacency matrix}, with rows and columns indexed by the cells of the network, such that the entry in row $i$ column $j$ corresponds to the number of input edges of that type from cell $j$ to cell $i$. If there is only one edge-type then there is only one adjacency matrix and the network is said {\em regular}. In this work, we consider a particular type of homogeneous networks - {\em with asymmetric inputs} - where each cell receives exactly one input edge of each type. 
Moreover, we consider that the networks are finite and, unless otherwise stated, that they are connected. 

There are several works in the literature that consider networks with asymmetric inputs. In \cite{AF10},  Agarwal and Field give a necessary and sufficient condition for the dynamical equivalence of two coupled cell networks with asymmetric inputs. Aguiar~\etal~\cite{AADF11} and Field~\cite{F14},~\cite{F17} consider the realization of heteroclinic cycles and networks for homogeneous networks with asymmetric inputs. In \cite{NRS16}, Nijholt~\etal\ introduce the concept of fundamental network for homogeneous networks with asymmetric inputs which reveals the hidden symmetries of the networks. More recently, Aguiar~\etal~\cite{ADS17} give a characterization of those fundamental networks.
In \cite{G14}, Ganbat gives the complete classification of codimension-one synchrony-breaking steady-state bifurcations for 1- input regular coupled cell networks. In \cite{NRS17}, Nijholt~\etal use projection blocks to describe the bifurcations of a particular type of 1-input regular coupled cell networks, with a ring and only one tail, which they call ring feed-forward networks.

Synchrony-breaking bifurcations are very relevant from the point of view of applications and the identification of the synchrony subspaces of a network is a crucial aspect in their study. For example, the synchrony-breaking bifurcation analysis, for a 1-input network with three-cells, conducted in Golubitsky~\etal~\cite{GPSZ09} has implications for certain models of the auditory system, in particular, models of the basilar membrane and attached hair bundles. More  concretely, they analyse how the periodic forcing of the first node in a chain of coupled identical systems, corresponding to a 1-input network with three-cells, whose internal dynamics is each tuned near a point of Hopf bifurcation, can lead naturally to successive amplification of the incoming signal.

Here we are interpreting networks with asymmetric inputs as {\em unweighted}, that is, there is no weight associated to the edges, which is equivalent to say that, each entry $ij$ of an adjacency matrix of such a network either is $1$ or $0$, whether or not there is a connection from cell $j$ to $i$. Nevertheless, they can be considered as a special type of weighted networks in which arrows of the same type have associated the same weight and different edge-types represent different weights. In this case, the non-zero entries of an adjacency matrix of the network is the weight value associated with the corresponding edge-type. The results in this work are naturally equally valid for this interpretation with weights. Networks with asymmetric inputs with weights are used, for example, in the modelling of animal locomotion. It is widely believed that animal locomotion is generated and controlled, in part, by a central pattern generator, which is a network of neurons in the central nervous system capable of producing rhythmic output. In \cite{GSBC98}, Golubitsky~\etal\ describe a network with asymmetric inputs which can generate the full range of phase relationships observed in the gaits of $2n$-legged animals, for all values of $n$, where connections of differing strength are represented by arrows with different markings. Synchrony-breaking is a mechanism for pattern generation in legged locomotion of animals. For the quadruped locomotion it is shown, in Buono and Golubitsky~\cite{BG01} and Stewart~\cite{S17}, that all quadruped gaits can occur as the first bifurcation from a fully synchronous equilibrium, for suitable parameters, for the eight-cell network proposed by Golubitsky and coworkers. In  \cite{IP17}, In and Palacios,  present a circuit realization of an animal (quadruped) robot controlled by a central pattern generator network of neurons, whose model and design are biologically-inspired by the work of Golubitsky and coworkers. Their hardware simulations show that the animal robot can indeed reproduce, via synchrony-breaking bifurcations, all the primary gates predicted by theory. Also based on the work developed by Golubitsky and coworkers, in \cite{RI06}, Righetti and Ijspeert construct a model of central pattern generator by means of a four-cell network of coupled oscillators with asymmetric inputs used to control crawling in a simulated humanoid robot. The work is part of a project whose purpose is to build a 54-degrees of freedom humanoid robot with the cognitive abilities of a child.

For a homogeneous network $\NN$, the set $I(\NN)$ of the subspaces that are left invariant by their adjacency matrices is a complete lattice with partial order given by inclusion and the meet and join operations given by the intersection and sum, respectively. As observed in Aguiar and Dias~\cite{AD14}, the set $V(\NN)$ of synchrony subspaces for $\NN$ is a subset of $I(\NN)$.
Moreover, as proved by Stewart~\cite{S07}, the set $V(\NN)$ forms a complete lattice taking the relation of inclusion: the meet operation is the intersection of subspaces, but apparently, there is no general form for the join operation. 
Let $\NN$ be a homogeneous network with asymmetric inputs. As in Aguiar and Dias~\cite{AD14}, if there is more than one edge-type in $\NN$ we consider the subnetworks for each edge-type. More concretely, for each edge-type $\EE_i$ in $\NN$, we consider the subnetwork $\NN_{\E_i}$ of $\NN$ with the same cells of $\NN$ and only the edges of type $\EE_i$. As proved in \cite{AD14}, the lattice of synchrony subspaces for $\NN$ is given by the intersection of the lattices of synchrony subspaces for the subnetworks  $\NN_{\E_i}$. We remark that these edge-type subnetworks $\NN_{\E_i}$ can be disconnected. If that is the case, they are given by the union of (connected) 1-input regular coupled cell networks with the same edge-type. Following the results in Aguiar and Ruan~\cite{AR12} for the join of networks, we show how to relate the lattice of synchrony subspaces for the disjoint union of two coupled cell networks with the same cell and edge-types from the lattice of synchrony subspaces for those networks. It remains then to describe the synchrony subspaces for the particular case of {\em 1-input regular coupled cell networks} - there is only one edge-type and each cell receives exactly one input.

Since the set $V(\NN)$ of synchrony subspaces for a network $\NN$ is a subset of the set $I(\NN)$ of the subspaces that are left invariant by their adjacency matrices, the natural join operation for the lattice $V(\NN)$ would be the sum but, as noted in \cite{S07}, not always the sum of two synchrony subspaces is a synchrony subspace. Thus, in general, the lattice $V(\NN)$ is not a sublattice of $I(\NN)$. Moreover, in general, it is not possible to define the join-irreducible set for the lattice of synchrony subspaces of a network and obtain the lattice through that join-dense set. 
The situation differs when one considers networks with asymmetric inputs. For this particular type of networks, as shown in Aguiar~\etal~\cite{AADF11}, the set of synchrony subspaces is closed under the sum operation. That is, the join operation for the lattice $V(\NN)$ of synchrony subspaces of networks $\NN$ with asymmetric inputs is the sum operation, and so, for this type of networks, $V(\NN)$ is a sublattice of the lattice $I(\NN)$.

In \cite{AD14}, Aguiar and Dias describe how to find a sum-dense set for the lattice of synchrony subspaces of a regular network, based on the eigenvectors and generalized eigenvectors of the associated adjacency matrix. By a {\em sum-dense set} ${\it m}$ for the lattice of synchrony subspaces it is meant a set of synchrony subspaces such that, although the sum of synchrony spaces in a subset  of ${\it m}$ may not be a synchrony space, every synchrony subspace in the lattice is given by the sum of the synchrony subspaces in a subset of ${\it m}$.
Here, we improve the results found in Aguiar and Dias~\cite{AD14} for the particular case of 1-input regular coupled cell networks.
Given the fact that the sum of synchrony subspaces is the join operation for the lattice of synchrony subspaces of a 1-input regular coupled cell network, instead of a sum-dense set, we are able to get the set of join-irreducible elements for the lattice of synchrony subspaces for this type of networks. The {\em set of join-irreducible elements} is given by all the synchrony subspaces that are not the join (sum) of any other two synchrony subspaces in the lattice. This set is {\em join-dense} for the lattice of synchrony subspaces, that is, the join (sum) of any subset of  join-irreducible synchrony subspaces is a synchrony subspace and every synchrony subspace is the join of a subset subset of  join-irreducible synchrony subspaces.
Taking into account the particular topology of 1-input regular coupled cell networks, we are able to characterize the eigenvectors and Jordan chains of the adjacency matrix $A$ of a 1-input regular coupled cell network. Since each join-irreducible synchrony subspace is generated by a basis of eigenvectors or generalized eigenvectors, we are then able to describe and characterize specifically  the set of join-irreducible elements for the lattice of synchrony subspaces of a 1-input regular network.  
To each synchrony subspace for a 1-input regular coupled cell network we can associate a balanced colouring: if we colour the cells of the network such that cells that are synchronized have the same colour then cells with a same colour receive their input connection from cells of a same colour.
We describe the possible patterns of balanced colourings associated to the join-irreducible elements and conclude about the possible patterns of balanced colourings associated to the other synchrony subspaces in the lattice of a 1-input regular coupled cell network.

The paper is organized as follows: Section~\ref{sec:Prelim} introduces concepts related with coupled cell networks, in particular, with homogeneous networks with asymmetric inputs. It also resumes definitions and results on coupled cell systems, and more specifically related with synchrony subspaces. At the end of the section we present basic definitions and results on complete lattices and on the lattice of synchrony subspaces for homogeneous networks.  Section~\ref{sec:lattice_union} contains our results on the lattice of synchrony subspaces for the union of identical-edge networks with the same cell-type and the same edge-type. More concretely, given a network $\NN = \NN_1 +  \NN_2$ that is the disjoint union of two regular networks $\NN_1$ and $\NN_2$ with the same cell and edge-type and such that their sets of cells have empty intersection, we describe the synchrony subspaces for $\NN$ in terms of the synchrony subspaces for $\NN_1$ and $\NN_2$. We get some useful remarks for the particular case of 1-input regular networks.
Section~\ref{sec:lattice_1_regular} includes our results about the lattice of synchrony subspaces for 1-input regular coupled cell networks. We identify the eigenvalues and associated eigenvectors for the adjacency matrix of a 1-input regular coupled cell network. This allows us to describe the join-irreducible elements in the lattice of synchrony subspaces for a 1-input regular coupled cell network. The set of the join-irreducible synchrony subspaces is join-dense for the lattice. 
We end the section with a description of the possible patterns of balanced colourings for the synchrony subspaces in the lattice. 
Finally, we end with some conclusions in Section~\ref{sec:conclusions}.

\section{Preliminary definitions and results} \label{sec:Prelim}

In this section we review and present some definitions related with homogeneous coupled cell networks with asymmetric inputs, coupled cell systems and synchrony subspaces that will be used throughout the text. 
More details on coupled cell networks and systems and synchrony subspaces can be found in Stewart~\etal~\cite{SGP03}, Golubitsky~\etal~\cite{GST05}, Golubitsky and Stewart\cite{GS05}, and references therein.  As the set of synchrony subspaces for a network is a lattice, we end with some basic definitions about complete lattices and with a result in Aguiar and Dias~\cite{AD14} about the lattice of synchrony subspaces for homogeneous networks. Details on complete lattices can be found, for example, in Davey and Priestley~\cite{DP90}. 

\subsection*{Homogeneous coupled cell networks with asymmetric inputs}

\begin{Def} \normalfont
A {\em coupled cell network} $\NN$ consists of a finite nonempty set $\CC$ of {\it cells} and a finite nonempty set $\EE= \{ (c,d):\ c,d \in \CC\}$ of {\em edges}, where each pair $(c,d)$ represents an edge from cell $d$ to cell $c$.  Moreover, it consists of a cell equivalence relation $\sim_{C}$ on $\CC$ and an edge equivalence relation $\sim_{E}$ on $\EE$ such that the {\it consistency condition} is satisfied: if $(c_1,d_1) \sim_{E} (c_2,d_2)$, then $c_1 \sim_{C} c_2$ and $d_1 \sim_{C} d_2$.
We write $\NN = (\CC,\EE,\sim_C, \sim_E)$. 
\END\end{Def} 

For an edge $(c,d)\in \EE$, the cells $c$ and $d$ are called, respectively, the {\it head} and  {\it tail} cell. The {\it input set} of cell $c$, denoted by $I(c)$, is given by the tail cells of all the edges with head $c$. 

A coupled cell network can be represented by a directed unweighted graph, where the cells are placed at vertices (nodes), the edges are depicted by directed arrows and the cell and edge equivalence relations are indicated, respectively, by different types of vertices and different types of edges in the graph.

\begin{Def} \normalfont  
Given a coupled cell network with set of cells $\CC$, we say there is a {\em directed path} connecting a sequence of cells $(c_0,c_1, \ldots ,c_{k-1},c_{k})$ of $\CC$, if there is an edge from $c_{j-1}$ to $c_{j}$, for $j \in \{1,...,k\}$.  If, for every $j \in \{1,...,k\}$, there is an edge from $c_{j-1}$ to $c_{j}$ or from $c_{j}$ to $c_{j-1}$, we say that there is an {\em undirected path} connecting the sequence of cells $(c_0,c_1, \ldots ,c_{k-1},c_{k})$.
A coupled cell network is {\em connected} if there is an undirected path between any two cells.
\END
\end{Def}

Unless otherwise stated, through the text we assume that a coupled cell network is connected.

\begin{Def} \normalfont \label{defhomognet} \label{defregnet}
A coupled cell network is said {\it homogeneous} if the cells are all identical and receive the same number of input edges per edge-type. A {\it regular} network is a homogeneous network with only one edge-type. For a homogeneous (regular) network, the total number of input edges per cell is the same for all cells and is called the {\em valency} of the network. 
\END\end{Def}

The definition of coupled cell network allows a cell to receive symmetric inputs, that is, more than one (unweighted) input edge of the same type. If that is not possible we say that the cells have asymmetric inputs.

\begin{Def} \normalfont
We say that a coupled cell network is a  {\em coupled cell network with asymmetric inputs} if each cell receives at most one input edge of each type. 
\END\end{Def} 

We consider homogeneous coupled cell networks with asymmetric inputs, which means that each cell receives exactly one edge of each type.

\begin{ex}  \normalfont 
The 7-cell network in Figure~\ref{fig:7cnetwork} is a homogeneous network with asymmetric inputs and has two edge-types. 
\END
\end{ex}

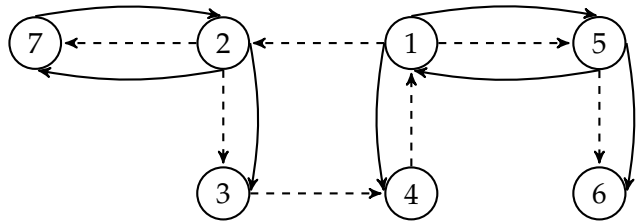
\begin{figure}[ht!]
\begin{center}
\vspace{-4mm}
\hspace{-4mm}
	{\small \begin{tikzpicture}	[->,>=stealth',shorten >=1pt,auto, node distance=1.5cm, thick,node/.style={circle,draw}]
	                 \node[node]	(1) at (-75mm, 4cm)  [fill=white] {$7$};
			\node[node]	(2) at (-5cm, 4cm)  [fill=white] {$2$};
			\node[node]	(3) at (-5cm, 2cm)  [fill=white] {$3$};
			\node[node]	(4) at (-25mm, 4cm)  [fill=white] {$1$};			
			\node[node]	(5) at (0cm, 4cm)  [fill=white] {$5$};			
			\node[node]	(6) at (0cm, 2cm)  [fill=white] {$6$};
			\node[node]	(7) at (-25mm, 2cm)  [fill=white] {$4$};
			
                                                 \path
				(2.south) edge [bend left=10]  node {} (1.south)
				(1.north) edge  [bend right=350]  node {} (2.north)
				(2.east) edge [bend right=350] node {} (3.east)
				(5.south) edge [bend left=10] node {} (4.south)
				(4.north) edge [bend right=350]  node {} (5.north)
				(4.west) edge [bend left=350]  node {} (7.west)
				(5.east) edge [bend right=350]  node {} (6.east)
				;
				\path
				 [dashed] 
				 (2) edge node {} (1) 
				(2) edge node {} (3)
				(3) edge node {} (7)
				(7) edge node {} (4)
				(4) edge node {} (2)
				(4) edge node {} (5)
				(5) edge node {} (6)
				;
		\end{tikzpicture}}
		\caption{A homogeneous coupled cell network with asymmetric inputs.}
		\label{fig:7cnetwork}
		\end{center}
\end{figure}

The coupling structure of a homogeneous network with set of cells $\CC=\{c_1, \ldots, c_n\}$ and $r$ edge-types $\EE_l$, $l=1,\ldots, r$, is given by $r$ adjacency matrices $A_l:=(a_{ij}^{(l)})$, of order $n \times n$, with rows and columns indexed by the cells in $\CC$, such that the entry $a_{ij}^{(l)}$ corresponds to the number of input edges of type $\EE_l$ from cell $c_j$ to cell $c_i$.

\begin{Def}\normalfont \label{def:intsym}
Let $\NN$ be a homogeneous network with set of cells $\CC$ and $S\subseteq \CC$. An {\it interior symmetry} of $\NN$ on $S$ is a permutation $\sigma$ on $\CC$ such that $\sigma$ fixes every element in $\CC\setminus S$, and, for each $c\in S$, $d\in \CC$, there is a bijection between edges $\left(\sigma(c),\sigma(d)\right)$ and $\left(c,d\right)$, which preserves the edges type. 
\END
\end{Def}

Let $\NN$ be an identical-cell network  with adjacency matrices  $A_l$, $l=1,\ldots, r$. Then, a permutation $\sigma$ is an interior symmetry of $\NN$ on $S$, if and only if 
\begin{equation}\label{eq:int_symm}
a_{cd}^{(l)}=a^{(l)}_{\sigma(c)\sigma(d)}, \quad \forall c \in S,\, d \in \CC, \, l=1, \ldots, r.
\end{equation}

\begin{ex}  \normalfont 
The 3-cell homogeneous network in Figure~\ref{fig:3cnetwork} has interior symmetry on $S=\{1, 3\}$. The 7-cell network in Figure~\ref{fig:7cnetwork} has no interior symmetry.
\END
\end{ex}

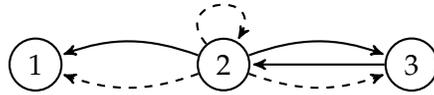
\begin{figure}[ht!]
\begin{center}
\vspace{-4mm}
\hspace{-4mm}
	{\small \begin{tikzpicture}	[->,>=stealth',shorten >=1pt,auto, node distance=1.5cm, thick,node/.style={circle,draw}]
	                 \node[node]	(1) at (-75mm, 4cm)  [fill=white] {$1$};
			\node[node]	(2) at (-5cm, 4cm)  [fill=white] {$2$};
			\node[node]	(3) at (-25mm, 4cm)  [fill=white] {$3$};			

                                                   \path
				(2) edge [out=160, in=20] node {} (1)
				(2) edge [out=20, in=160] node {} (3)
				(3) edge node {} (2)
								;
				\path
				 [dashed] 
				(2) edge [out=200, in=340] node {} (1)
				(2) edge [out=340, in=200] node {} (3)
				(2) edge [loop, out=130, in=60, looseness=5] node [above right, pos=.75]  {}  (2)				;
		\end{tikzpicture}}
		\caption{A homogeneous coupled cell network with interior symmetry on $S=\{1, 3\}$.}
		\label{fig:3cnetwork}
		\end{center}
\end{figure}

In the following definition, given a homogeneous network $\NN$ with $r$ edge-types, we consider $r$ subnetworks of $\NN$, one for each edge-type, with the set of cells of $\NN$ but only the edges of that type.

\begin{Def} \normalfont \label{def:subnet}
Let $\NN = \left(\CC, \EE, \sim_{\CC},\sim_{\EE}\right)$ be a homogeneous coupled cell network and $\EE_1, \ldots, \EE_r$ the $\sim_{\EE}$-equivalence classes. For $l=1, \ldots, r$, we define the {\em identical-edge subnetwork $\NN_{\E_{l}}$ } of $\NN$, for the edge-type $\EE_{l}$, as $\NN_{\E_{l}} = \left(\CC, \EE_{l} , \sim_{\CC}, \{ \EE_l\}\right)$.
\END\end{Def}

Note that, for each $l=1, \ldots, r$, the subnetwork $\NN_{\E_{l}}$ may not be connected. In that case, it is the disjoint union of connected networks. The disjoint union of regular coupled cell networks with the same cell and edge-type is defined in the same way as the disjoint union of graphs.

\begin{Def} \normalfont \label{def:unionnet}
Let $\NN_i = \left(\CC_i, \EE_i, \sim_{\CC_i} = \{\CC_i\},\sim_{\EE_i} = \{\EE_i\}\right)$, $i=1,2$, be two regular coupled cell networks with the same cell and edge-type and such that $\CC_1 \cap \CC_2 = \emptyset$ (and thus  $\EE_1 \cap \EE_2 = \emptyset$). The {\em disjoint union} $\NN = \NN_1 +  \NN_2$ of the networks $\NN_1$ and $\NN_2$ is the network $\NN = \left(\CC_1 \cup \CC_2, \EE_1 \cup \EE_2, \{\CC_1 \cup \CC_2\},\{\EE_1 \cup \EE_2\} \right)$.
\END\end{Def}

If $\NN$ is homogeneous, each subnetwork $\NN_{\E_{l}}$, $l=1, \ldots, r$,  of $\NN$ is a regular network or a union of regular networks.  If, in addition, $\NN$ has asymmetric inputs then, for each edge-type $\EE$, the subnetwork $\NN_{\E_l}$ is a {\em $1$-input regular coupled cell network} or a union of $1$-input regular coupled cell networks (with the same cell and edge-types).

See Ganbat~\cite{G14} for a complete classification of codimension-one synchrony-breaking steady-state bifurcations in $1$-input regular coupled cell networks.

\begin{ex}  \normalfont 
Consider the homogeneous network with asymmetric inputs $\NN$ in Figure~\ref{fig:7cnetwork}. Let $\EE_1$ be type of the solid edges and $\EE_2$ be type of the dashed edges. The subnetwork $\NN_{\E_1}$ of $\NN$, depicted in Figure~\ref{fig:subN1}, is the union of two $1$-input regular coupled cell networks. The subnetwork $\NN_{\E_2}$  is a $1$-input regular coupled cell network, see Figure~\ref{fig:subN2}.
\END
\end{ex}

\begin{figure}[ht!]
\begin{center}
\vspace{-4mm}
\hspace{-4mm}
	{\small \begin{tikzpicture}[->,>=stealth',shorten >=1pt,auto, node distance=1.5cm, thick,node/.style={circle,draw}]
	      \node[node]	(1) at (-75mm, 4cm)  [fill=white] {$7$};
			\node[node]	(2) at (-5cm, 4cm)  [fill=white] {$2$};
			\node[node]	(3) at (-5cm, 2cm)  [fill=white] {$3$};
			\node[node]	(4) at (-25mm, 4cm)  [fill=white] {$1$};			
			\node[node]	(5) at (0cm, 4cm)  [fill=white] {$5$};			
			\node[node]	(6) at (0cm, 2cm)  [fill=white] {$6$};
			\node[node]	(7) at (-25mm, 2cm)  [fill=white] {$4$};
	
                                                   \path
				(2.south) edge [bend left=10]  node {} (1.south)
				(1.north) edge  [bend right=350]  node {} (2.north)
				(2.east) edge [bend right=350] node {} (3.east)
				(5.south) edge [bend left=10] node {} (4.south)
				(4.north) edge [bend right=350]  node {} (5.north)
				(4.west) edge [bend left=350]  node {} (7.west)
				(5.east) edge [bend right=350]  node {} (6.east)
				;
		\end{tikzpicture}}
		\caption{Subnetwork $\NN_{\E_1}$ of the network $\NN$ in Figure~\ref{fig:7cnetwork}, with $\EE_1$ the solid edge-type.}
		\label{fig:subN1}
		\end{center}
\end{figure}
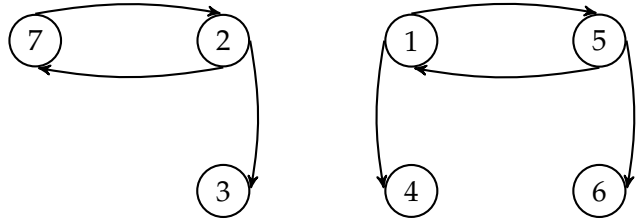

\begin{figure}[ht!]
\begin{center}
\vspace{-4mm}
\hspace{-4mm}
	{\small \begin{tikzpicture}	[->,>=stealth',shorten >=1pt,auto, node distance=1.5cm, thick,node/.style={circle,draw}]
	                 \node[node]	(1) at (-75mm, 4cm)  [fill=white] {$7$};
			\node[node]	(2) at (-5cm, 4cm)  [fill=white] {$2$};
			\node[node]	(3) at (-5cm, 2cm)  [fill=white] {$3$};
			\node[node]	(4) at (-25mm, 4cm)  [fill=white] {$1$};			
			\node[node]	(5) at (0cm, 4cm)  [fill=white] {$5$};			
			\node[node]	(6) at (0cm, 2cm)  [fill=white] {$6$};
			\node[node]	(7) at (-25mm, 2cm)  [fill=white] {$4$};

				\path
				 [dashed] 
				 (2) edge node {} (1) 
				(2) edge node {} (3)
				(3) edge node {} (7)
				(7) edge node {} (4)
				(4) edge node {} (2)
				(4) edge node {} (5)
				(5) edge node {} (6)
				;
		\end{tikzpicture}}
		\caption{Subnetwork $\NN_{\E_2}$ of the network $\NN$ in Figure~\ref{fig:7cnetwork}, with $\EE_2$ the dashed edge-type.}
		\label{fig:subN2}
		\end{center}
\end{figure}

\begin{Def} \normalfont  
Given a coupled cell network with set of cells $\CC$, a directed path $(c_0,c_1, \ldots ,c_{k-1},c_{k})$, with $c_i \in \CC$, $i=0,\ldots,k$, such that $c_0 = c_k$ is called a {\em ring} (or loop). In particular, if $k=1$, a ring is a {\em self-loop}.
\END
\end{Def}

Ring networks have been studied, for example, in Ganbat~\cite{G14} and Moreira ~\cite{M14}.

\begin{Def} \normalfont  
A network for which there is a cell $c$ such that, for any other cell $d$, there is exactly one directed path from $c$ to $d$ is called a {\em directed rooted tree}. The cell $c$ is called the {\em root}. The cells $d$ with no outgoing connection are called the {\em leafs}. For each leaf $d$, the directed path from the root $c$ to $d$ is called a {\em tail}. A connected subgraph of a tree is a {\em subtree}.
\END
\end{Def}

\begin{rmk} \normalfont \label{rmk:ring+trees} 
A  {\em $1$-input regular coupled cell network} with $n$ cells either is a ring (in particular, a self-loop) or the coalescence of a ring with the disjoint union of a finite number $s$ of directed rooted trees $\TT_i$, for $i =1, \dots, s$, such that the root $r_i$ of each rooted tree $\TT_i$ merges with a different cell in the ring. Remember that a {\em coalescence} of two graphs $G_1$ and $G_2$ is a graph obtained from the disjoint union of $G_1$ and $G_2$ by identifying a vertex of $G_1$ with a vertex of $G_2$, that is, by merging one vertex from each graph into a single vertex. 
\END
\end{rmk}

\begin{ex}  \normalfont 
The $1$-input regular coupled cell network in Figure~\ref{fig:subN2} is  the coalescence of the ring $(1,2,3,4,1)$ with two directed rooted trees, one with root at cell $1$ and the other with root at cell $2$. The $1$-input regular coupled cell network in Figure~\ref{fig:1-inputNet} is the coalescence of the ring $(1,2,3,1)$ with two directed rooted trees, one with root at cell $2$ and the other with root at cell $3$.
\END
\end{ex}

{\tiny
\begin{figure}[ht!]
\begin{center}
\vspace{-4mm}
\hspace{-4mm}
	{\small \begin{tikzpicture}[->,>=stealth',shorten >=1pt,auto, node distance=1.5cm, thick,node/.style={circle,draw}]
	                 \node[node]  	(1) at (0cm, 0cm)  [fill=green] {$1$};
			\node[node]	(2) at (2cm, 1cm)  [fill=green] {$2$};
			\node[node]	(3) at (2cm, -1cm)  [fill=green] {$3$};
			\node[node]	(4) at (4cm, 2cm)  [fill=orange] {$4$};	
			\node[node]	(8) at (4cm, 1cm)  [fill=green] {$8$};	
			\node[node]	(9) at (6cm, 1cm)  [fill=green] {$9$};						
			\node[node]	(10) at (4cm, -1cm) [fill=green] [circle,draw, label=center:10]  {\phantom{0}};
			\node[node]	(5) at (6cm, 3cm)  [fill=yellow] {$5$};
			\node[node]	(6) at (6cm, 2cm)  [fill=yellow] {$6$};
                 	\node[node]	(11) at (6cm, -1cm) [fill=orange] [circle,draw, label=center:11]  {\phantom{0}};	
                 	\node[node]	(7) at (8cm, 2cm)  [fill=magenta] {$7$};
                 	\node[node]	(12) at (8cm, 0cm) [fill=brown] [circle,draw, label=center:12]  {\phantom{0}};
	                \node[node]	(13) at (8cm, -1cm) [fill=yellow] [circle,draw, label=center:13]  {\phantom{0}};
	                 \node[node]	(14) at (10cm, -1cm) [fill=magenta] [circle,draw, label=center:14]  {\phantom{0}};
	                  \node[node]	(15) at (12cm, -1cm) [fill=blue!40] [circle,draw, label=center:15]  {\phantom{0}};
                                                   \path
	                 (1) edge node {} (2)
	                 (2) edge node {} (3)
	                 (2) edge node {} (8)
	                 (8) edge node {} (9)
	                 (3) edge node {} (1) 
	                 (2) edge node {} (4)   
	                 (4) edge node {} (5) 
	                 (6) edge node {} (7) 
	                 (4) edge node {} (6) 
	                 (3) edge node {} (10) 
	                 (10) edge node {} (11) 
	                 (11) edge node {} (12) 
	                 (11) edge node {} (13) 
	                 (13) edge node {} (14) 
	                 (14) edge node {} (15) 
				;
		\end{tikzpicture}}
		\caption{A $1$-input regular coupled cell network that is formed by the ring $(1,2,3,1)$ and two directed rooted trees, one with root at cell $2$ and the other with root at cell $3$.}
		\label{fig:1-inputNet}
		\end{center}
\end{figure}
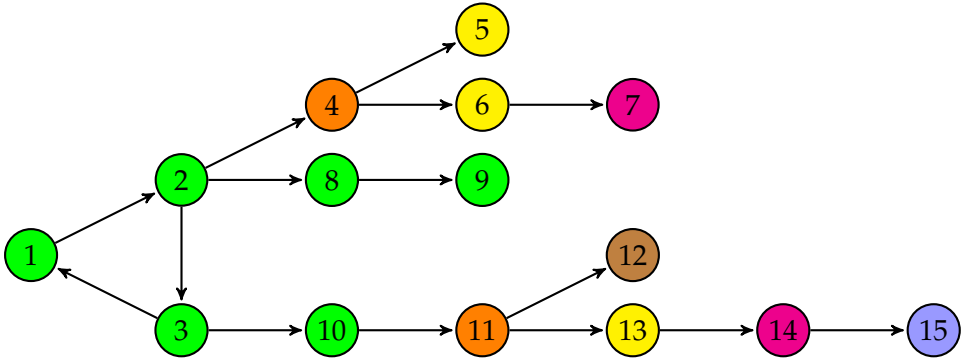
}

We remark that a $1$-input regular coupled cell network, with $n >1$ cells, whose ring consists of a cell with a self-loop is a particular case of an auto-regulation feed-forward neural network, see Aguiar \etal~\cite{ADF17}.

Following Aguiar~\etal~\cite{ADS17}, we define next the depth of a $1$-input regular network as the maximal distance to the ring of any cell out of the ring.  We start by defining the depth of a directed rooted tree.

\begin{Def} \normalfont  
Let $\TT$ be a directed rooted tree with root $r$ and $\LL$ the set of leaf cells in $\TT$. We define the {\em depth of $\TT$} as zero if $\LL = \emptyset$, otherwise 
$$
depth(\TT) := max \{|(r,l)| : l \in \LL \},
$$
where $|(r, l)|$ is the number of edges in the tail in $\TT$ from $r$ to $l$.

Let $\NN$ be a  $1$-input regular coupled cell network. If $\NN$ is a ring then we define the {\em depth of $\NN$} as zero, otherwise 
$$
depth(\NN) := max\{ depth(\TT_i):  i \in \{1, \dots, s\}\},
$$
with $\TT_i$, for $i =1, \dots, s$, the $s$ directed rooted trees of $\NN$, as in Remark~\ref{rmk:ring+trees}.
\END
\end{Def}

\begin{ex}  \normalfont 
The depth of the $1$-input regular network in Figure~\ref{fig:subN2} is $2 = max\{ 1,2\}$ and the depth of the $1$-input regular network in Figure~\ref{fig:1-inputNet} is $5 = max\{ 2,3,5\}$.
\END
\end{ex}

\subsection*{Coupled cell systems}

Let $\CC = \{ 1, \ldots, n\}$ be the set of cells of a coupled cell network.
To each cell $c \in \CC$ we associate a {\it cell phase space} $P_c$ which is assumed to be a finite-dimensional real vector space, say $\RR^k$ for some $k >0$. When cells are equivalent, as in the case of homogeneous networks, the corresponding phase spaces are identified canonically. 
 The {\it total phase space} $P = \prod_{c \in {\cal C}} P_c$ 
is the direct product of the cell phase spaces
and we employ the coordinate system $x=(x_c)_{c \in \cal C}$ on $P$. 

To each coupled cell network we associate a class of admissible continuous coupled cell systems. 
Given a network $\NN$ and a fixed choice of the total phase space $P$, the class of the ordinary differential equations, $\dot X = F(X),\ X \in P$, compatible with the structure of the network $\NN$ are such that the $j$th coordinate of the vector field, defining the equation associated with cell $j$, has the form
\[
\dot{x}_j = f_j \left( x_j; x_{i_1}, \ldots, x_{i_m}\right)
\]
where the first argument $x_j$ in $f_j$ represents  the internal dynamics of the cell $j$, which has $m$ input-edges, and each of the remaining variables $x_{i_p}$, $p=1,\ldots,m$, represents an edge from 
cell $i_p$ to cell $j$. Thus $x_j \in P_j, x_{i_p} \in P_{i_p}, \ p=1, \ldots,m$  and we assume $f_j:\ P_j \times P_{i_1} \times \cdots \times P_{i_m} \to P_{j}$ is smooth. For homogeneous networks, since the cells are all identical, the internal dynamics of the cells is the same for all cells, that is, $f_j = f$, for all $j=1,\ldots,n$. 
The vector fields $F$ are said {\em admissible} by $\NN$. 

\begin{ex}\label{ex:admissvf}
\normalfont
Consider the asymmetric homogeneous network $\NN$ in Figure~\ref{fig:7cnetwork}. The coupled cell systems associated to $\NN$ satisfy
\[
\begin{array}{l}
\dot x_1 =f(x_1;x_5,x_4) \\
\dot x_2 =f(x_2;x_7,x_1) \\
\dot x_3 =f(x_3;x_2,x_2) \\
\dot x_4 =f(x_4;x_1,x_3) \\
\dot x_5 =f(x_5;x_1,x_1) \\
\dot x_6 =f(x_6;x_5,x_5) \\
\dot x_7 =f(x_7;x_2,x_2) \\
\end{array},
\]
where $f:\ \left(\RR^k\right)^7 \to \RR^k$ is smooth. For each equation $j =1,\ldots,7$, the first argument of $f$ is the internal dynamics of cell $j$, the second and third arguments are the variables corresponding to the inputs of solid and dashed edge-types, respectively, for the cell $j$.
\END
\end{ex}

\subsection*{Synchrony subspaces}

The structure of a network imposes the existence of certain flow-invariant subspaces for any coupled cell system compatible with that structure. These are called synchrony subspaces.

\begin{Def} [\cite{GST05} Definition 4.2] \normalfont \label{def:espacosincronia}
Consider a network $\NN$ with total phase space $P$.  A {\em polydiagonal subspace} is a subspace of $P$ characterized by a set of equalities of cell coordinates.
A  {\em synchrony subspace} for the network $\NN$ is a polydiagonal subspace of $P$ which is flow-invariant under all the vector fields on $P$ that are admissible by $\NN$. 
\END 
\end{Def}

\begin{ex}\label{ex:synsub}
\normalfont
Consider again the asymmetric homogeneous network $\NN$ in Figure~\ref{fig:7cnetwork} and the general form of the coupled cell systems on $P= \left(\RR^k\right)^7$ admissible by $\NN$ given in Example~\ref{ex:admissvf}. The polydiagonal subspace $\{ x \in P:\ x_1 = x_3\}$ is a synchrony subspace for $\NN$. Note that, if we consider an initial condition on $P$ such that $x_3 = x_7$ then the equations for $\dot x_3$ and $\dot x_7$ coincide and the trajectory satisfies the equality $x_3 = x_7$ for all time. The polydiagonal subspace $\{ x \in  P:\ x_1 = x_3 = x_7\}$ is not a synchrony subspace for $\NN$ but the polydiagonal subspace $\{ x \in  P:\ x_1 = x_3 = x_7,\ x_2=x_4=x_5\}$ is a synchrony subspace for $\NN$.
\END
\end{ex}

The concept of synchrony subspace for a network is closely related to that of balanced equivalence relation on the network set of cells. 

\begin{Def} [\cite{GST05} Definition 4.1] \normalfont \label{def:balanced}  
An equivalence relation $\bowtie$ on the 
 set of cells $\CC$ of a network $\NN$ is 
{\em balanced} if for every $c,d \in \CC$ with $c \bowtie d$, 
there exists an isomorphism between the input sets, 
$I(c)$ and $I(d)$,  of $c$ and 
$d$, respectively, say $\beta:\ I(c) \to I(d)$,  
preserving the arrow equivalence 
relation and such that for all $i \in I(c)$, the tail cells of $i$ and 
$\beta(i)$ are in the same $\bowtie$ class.
 \END 
\end{Def}

Given an equivalence relation $\bowtie$ on the set of cells $\CC$ a network $\NN$  and 
a choice of the total phase space $P$, define 
the {\em polydiagonal} subspace
\[
\Delta_{\bowtie} = \left\{ {\bf x} \in P:\ x_c = x_d \mbox{ whenever } 
c \bowtie d,\quad \forall c,d \in \CC\right\}\, .
\]

We have then the result of \cite{GST05} that relates the synchrony spaces of a network with the balanced equivalence relations on the set of cells of the network. 

\begin{thm}[\cite{GST05} Theorem 4.3] \label{thm:main}
Consider  a network $\NN$, an equivalence relation $\bowtie$ on the network 
set of cells $\CC$ and a choice $P$ of the total phase space. We have that, 
$\Delta_{\bowtie}$ is a synchrony subspace if and only if 
$\bowtie$ is balanced.
\end{thm}

Following Golubitsky~\etal~\cite{GST05}, we can visualize graphically a balanced equivalence relation, and so a synchrony subspace, for a network by a {\em balanced colouring} of the cells of the network. More concretely, given an equivalence relation on the network set of cells, if we colour the cells such that cells in the same class have the same colour then the equivalence relation is balanced if and only if, taking any two colours $r_1$ and $r_2$, for each edge-type, all cells with colour $r_1$ receive the same number of input edges of that type from the cells of colour $r_2$.
For the specific case of networks with asymmetric inputs, the colouring is balanced if and only if, for each edge-type, cells with a same colour receive their input connection of that edge-type from cells of a same colour. 

Theorem 5.2 of Golubitsky~\etal~\cite{GST05} shows that, associated to every synchrony subspace of a network $\NN$ there is always a network $Q$, called the {\em quotient network}, such that the restrictions of the admissible vector fields for $\NN$ to the synchrony subspace are the admissible vector fields of the quotient network $Q$. Given a synchrony subspace of $\NN$ and the corresponding balanced colouring, the quotient network $Q$ is obtained by the identification of the cells with the same colour (the cells that are synchronized) and projection of the edges, preserving the cell types and the edge types. For a more formal definition of quotient network, see Golubitsky~\etal~\cite{GST05}.

\begin{ex}\label{ex:quonet}
\normalfont
Consider the network $\NN$ in Figure~\ref{fig:7cnetwork} and the synchrony subspace $\Delta_{\bowtie} = \{ x \in  P:\ x_1 = x_3 = x_7,\ x_2=x_4=x_5\}$ for $\NN$. The corresponding balanced equivalence relation is $\bowtie = \{ \{1,3,7\},\ \{2,4,5\},\ \{6\}\}$. 
Colour the three equivalence classes by green, orange and yellow, respectively. As noted above, this is a balanced colouring. In fact, each green cell receives one solid edge from an orange cell and one dashed edge from an orange cell. The same for each yellow cell. Each orange cell receives one solid edge from a green cell and one dashed edge from a green cell.  See Figure~\ref{fig:7cnetwork_col} (left). 

The quotient network $Q$, presented in Figure~\ref{fig:7cnetwork_col} (right), is obtained as we explain next. Each colour (equivalence class) corresponds to a cell in $Q$: the green cells $1,3,7$ of $\NN$ are identified as the green cell $1$ in $Q$, the orange cells $2,4,5$ of $\NN$ are identified as the orange cell $2$ in $Q$, and the yellow cell $6$ of $\NN$ projects into the yellow cell $6$ of $Q$. As for the edges, since each green cell in $\NN$ receives one solid edge and one dashed edge from an orange cell, the green cell $1$ in $Q$ receives one solid edge and one dashed edge from the orange cell $2$ in $Q$. Analogously for the yellow cell. Since each orange cell in $\NN$ receives one solid edge and one dashed edge from a green cell, the orange cell $2$ in $Q$ receives one solid edge and one dashed edge from the green cell $1$ in $Q$. 

\END
\end{ex}

\begin{figure}[ht!]
\begin{center}
\vspace{-4mm}
\hspace{-4mm}

	{\small \begin{tikzpicture}	[->,>=stealth',shorten >=1pt,auto, node distance=1.5cm, thick,node/.style={circle,draw}]
	                 \node[node]	(1) at (-75mm, 4cm)  [fill=green] {$7$};
			\node[node]	(2) at (-5cm, 4cm)  [fill=orange] {$2$};
			\node[node]	(3) at (-5cm, 2cm)  [fill=green] {$3$};
			\node[node]	(4) at (-25mm, 4cm)  [fill=green] {$1$};			
			\node[node]	(5) at (0cm, 4cm)  [fill=orange] {$5$};			
			\node[node]	(6) at (0cm, 2cm)  [fill=yellow] {$6$};
			\node[node]	(7) at (-25mm, 2cm)  [fill=orange] {$4$};
			
			 \node[node]	(q1) at (25mm, 3cm)  [fill=green] {$1$};
			\node[node]	(q2) at (5cm, 3cm)  [fill=orange] {$2$};
			\node[node]	(q6) at (75mm, 3cm)  [fill=yellow] {$6$};

                                                   \path
				(2.south) edge [bend left=10]  node {} (1.south)
				(1.north) edge  [bend right=350]  node {} (2.north)
				(2.east) edge [bend right=350] node {} (3.east)
				(5.south) edge [bend left=10] node {} (4.south)
				(4.north) edge [bend right=350]  node {} (5.north)
				(4.west) edge [bend left=350]  node {} (7.west)
				(5.east) edge [bend right=350]  node {} (6.east)
				  (q1.north) edge  [bend right=350]  node {} (q2.north)
                                   (q2.north) edge  [bend right=350]  node {} (q6.north)
                                   (q2) edge [out=160, in=20] node {} (q1)
                                 
				;
				\path
				 [dashed] 
				 (2) edge node {} (1) 
				(2) edge node {} (3)
				(3) edge node {} (7)
				(7) edge node {} (4)
				(4) edge node {} (2)
				(4) edge node {} (5)
				(5) edge node {} (6)
				 (q1.south) edge [bend left=350] node {} (q2.south)
				 (q2.south) edge [bend left=350] node {} (q6.south)
                                 (q2) edge [out=200, in=340] node {} (q1)
				;
		\end{tikzpicture}}
		\caption{A balanced colouring for the network in Figure~\ref{fig:7cnetwork} (left) and the corresponding quotient network (right).}
		\label{fig:7cnetwork_col}
		\end{center}
\end{figure}
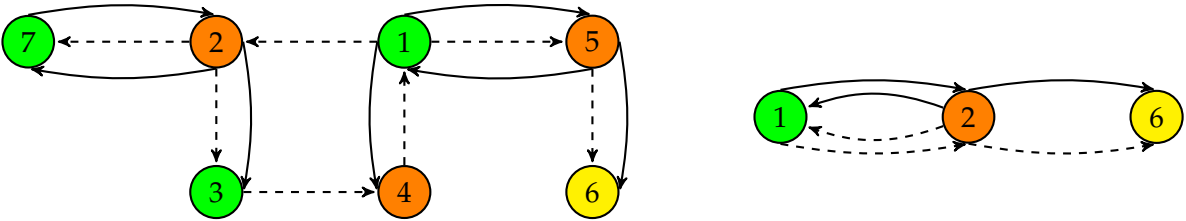

Let $\Delta_{\bowtie}$ be a synchrony subspace for a network $\NN$ and consider the corresponding quotient network $Q$. Let $\Delta_{\candy}$ be another synchrony subspace for $\NN$ such that $\Delta_{\candy} \subset \Delta_{\bowtie}$. We define the {\em restriction}, associated to $\Delta_{\bowtie}$, of $\Delta_{\candy}$ to $Q$, which we denote by $R\left(\Delta_{\candy}\right)$, as the polydiagonal subspace of the phase space of $Q$ defined by the equality conditions in the definition of  $\Delta_{\candy}$, that do not belong to the set of equality conditions that define $\Delta_{\bowtie}$, replacing the coordinate of each cell by the coordinate of its equivalence class by the balanced colouring. 
Let $\Delta_{\bowtie_Q}$ be a synchrony subspace for  the quotient network $Q$. We define the {\em lift}, associated to $\Delta_{\bowtie}$, of $\Delta_{\bowtie_Q}$ to $\NN$, which we denote by $L\left(\Delta_{\bowtie_Q}\right)$, as the polydiagonal subspace of the phase space of $\NN$ defined by the equality conditions in the definition of $\Delta_{\bowtie}$ together with the equality conditions in the definition of $\Delta_{\bowtie_Q}$ where the coordinate of each equivalence class by the balanced colouring is replaced by the coordinate of any cell in the class.
Note that the restriction and lift are inverse operations, that is, we have $L\left( R\left( \Delta_{\bowtie} \right) \right) = \Delta_{\bowtie}$ and $R\left( L\left( \Delta_{\bowtie_Q} \right) \right) = \Delta_{\bowtie_Q}$.

The following proposition corresponds to Proposition 2.9 in Aguiar and Ruan~\cite{AR12} rewritten in an equivalent way and in terms of synchrony subspaces instead of balanced equivalence relations.

\begin{proposition} [\cite{AR12} Proposition 2.9] \label{prop:rest_lift}
Let $\Delta_{\bowtie}$ be a synchrony subspace for a network $\NN$, $Q$ the associated quotient network and consider the notation above. We have:
\begin{itemize}
\item [(a)] If $\Delta_{\candy}$ is a synchrony subspace for $\NN$ such that $\Delta_{\candy} \subset \Delta_{\bowtie}$ then the polydiagonal $R\left(\Delta_{\candy}\right)$ is a synchrony subspace for $Q$.
\item [(b)] If $\Delta_{\bowtie_Q}$ is a synchrony subspace for $Q$  then the polydiagonal $L\left(\Delta_{\bowtie_Q}\right)$ is a synchrony subspace for $\NN$.
\end{itemize}
\end{proposition}

\begin{ex}\label{ex:restlift}
\normalfont
Consider the network $\NN$ in Figure~\ref{fig:7cnetwork} and its synchrony subspaces $\Delta_{\bowtie}= \{ x \in  P:\ x_1 = x_3 = x_7,\ x_2=x_4=x_5\}$ and $\Delta_{\candy}= \{ x \in  P:\ x_1 = x_3 = x_6 = x_7,\ x_2=x_4=x_5\}$. Consider the quotient network $Q$, in Figure~\ref{fig:7cnetwork_col}, associated to $\Delta_{\bowtie}$ and denote by $P_Q$ the phase space for $Q$. The restriction $R\left(\Delta_{\candy}\right) = \{ x \in  P_Q:\ x_1 = x_6\}$ is a synchrony subspace for $Q$. Consider the synchrony subspace $\Delta_{\bowtie_Q}= \{ x \in  P_Q:\ x_1 = x_2\}$ for $Q$. The lift $L\left(\Delta_{\bowtie_Q}\right) = \{ x \in  P:\ x_1 =x_2= x_3 =x_4= x_5 = x_7 \}$ is a synchrony subspace for $\NN$.
\END
\end{ex}

The following result, which is a consequence of Theorem 4.3 in Golubitsky~\etal~\cite{GST05} as explained in \cite{AD14}, gives a much simpler necessary and sufficient condition for a polydiagonal to be a synchrony subspace.

\begin{cor}[\cite{AD14} Corollary 2.11] \label{cor:ADlinear}
Let $\NN$ be a coupled cell network with set of cells $\CC$. For any choice of the total phase space $P$, a polydiagonal subspace is a synchrony subspace for $\NN$ if and only if it is flow-invariant under {\em all linear} admissible vector fields choosing the cell phase spaces to be $\RR$.
\end{cor}

\begin{rmk}[\cite{AD14} Remark 2.12]\normalfont \label{rmk:moreonlinear} For an $n$-cell homogeneous network $\NN$, the linear admissible vector fields, assuming the cell phase spaces to be $\RR$, are generated by the identity map on $\RR^n$ and the linear maps on $\RR^n$ associated to the network adjacency matrices. It follows then that, assuming the cell phase spaces to be $\RR$, a polydiagonal  subspace is a synchrony subspace for $\NN$ if and only if it is left invariant by the network adjacency matrices.
\END
\end{rmk}

As proved by Stewart~\cite{S07}, the set of synchrony subspaces associated with a coupled cell network, taking the partial relation of inclusion $\subseteq$, is a complete lattice. 

\subsection*{Complete lattices}
Following Section 3.1 of Aguiar and Dias~\cite{AD14}, we present basic definitions and results on complete lattices.

Given a partially ordered set $L$ with a binary relation $\geq$ and a subset $M \subseteq L$, an element $a$ of $L$ is an {\em upper bound} of $M$ if $a \geq b$ for all $b \in M$. Further, an upper bound $a$ of $M$ is said to be the {\em least upper bound} of $M$ if every upper bound $a^\prime$ of $M$ satisfies $a^\prime \geq a$. Dually, we define {\em lower bound} and {\em greatest lower bound}. 

\vspace{2mm}

Now recall that a {\em lattice} is a partially ordered set $L$ such that every 
pair of elements $a, b \in L$ has a {\em unique least upper bound or join},  
denoted by $a \vee b$, and a {\em unique greatest lower bound or meet}, denoted by $a\wedge b$. 

\vspace{2mm}

A {\em complete lattice} is a lattice where every 
subset $M \subseteq L$ has a unique least upper bound or join, and  
a unique greatest lower bound or meet. A complete lattice has a {\em top} (maximal) element, denoted $\top$, and a {\em bottom} (minimal) element, 
denoted $\perp$. Observe that every finite lattice is complete, see~\cite[Corollary 2.12]{DP90}.

\begin{ex}[\cite{AD14} Example 3.1] \normalfont \label{ex:sublattice} 
Given a linear map $A:\ \RR^n \to \RR^n$, the set of $A$-invariant subspaces is a lattice (considering the partial order $\subseteq$) with the meet operation corresponding to the intersection and the join operation given by the sum. (The sum corresponds to the subspace generated by the union.) The top element is $\RR^n$ and the bottom element is $\{ 0\}$. Moreover, that lattice is either finite or uncountably infinite. See for example Gohberg~\etal~\cite[Proposition 2.5.4]{GLR86}.
\END
\end{ex}

A {\em sublattice} $S_L$ of a lattice $L$ is a nonempty subset of $L$ that is a lattice with the same meet and join operations as $L$. That is,
\[
x \in S_L \mbox{ and } y \in S_L \Longrightarrow x \vee y \in S_L \mbox{ and } x \wedge y \in S_L\, .
\]

An element $a$ in a lattice $L$ is said to be {\em join-irreducible} if it is not the bottom element (in case $L$ has a bottom element) and if $a = x \vee y$ then $a = x$ or $a=y$, for all $x,y \in L$.
A {\em meet-irreducible} element is defined dually. See for example~Davey~\cite[Definition 8.7]{DP90}. 

\vspace{2mm}

A subset $Q$ of a lattice $L$ is said to be {\em join-dense} in $L$ if $L = \{\bigvee J_x | J_x \subseteq Q \}$.
The dual of join-dense is {\em meet-dense}. See for example~\cite[Definition 2.34]{DP90}. 
\newline

A partially ordered set P satisfies DCC (the descending chain condition) provided that there is no infinite decreasing sequence in P, equivalently, provided that each non-empty subset of P has a minimal element. Trivially every finite partially ordered set satisfies DCC.

Denote the set of join-irreducible elements of $L$ by $\mathcal{J}(L)$ and the set of meet-irreducible elements by $\mathcal{M}(L)$. 

\begin{thm}[\cite{DP90} Theorem 2.46 (i)] \label{thm:join-dense}
Let $L$ be a lattice that satisfies (DCC). Then $\mathcal{J}(L)$ and, more generally, any subset $Q$ which contains $\mathcal{J}(L)$ is join-dense in $L$. Dually, $\mathcal{M}(L)$ and, more generally, any subset $Q$ which contains $\mathcal{M}(L)$ is meet-dense in $L$. 
\end{thm}

\subsection*{Lattice of synchrony subspaces for homogeneous networks}

The following corollary translates the result in Corollary 4.3 of Aguiar and Dias~\cite{AD14}, in terms of synchrony subspaces, and states that a polydiagonal  is a synchrony subspace for a homogeneous network if and only if it is a synchrony subspace for all its identical-edge subnetworks.

\begin{cor} [\cite{AD14} Corollary 4.3]  \label{cor:hom}
 Let $\NN = \left(\CC, \EE, \sim_{\CC},\sim_{\EE}\right)$ be a homogeneous coupled cell network. For $\EE_j$, $j=1, \ldots, l$, the $\sim_{\EE}$-equivalence classes of $\NN$, consider the identical-edge subnetworks $\NN_{\E_{j}}$ and the corresponding lattice of synchrony subspaces $V_{\N_{\E_{j}}}$.
Then the following holds:\\ 
\[
\displaystyle V_{\N} =  \bigcap_{j=1}^{l} V_{\N_{\E_{j}}}\, .
\]
\end{cor}

As remarked in \cite{AD14}, one way to implement an efficient algorithm to obtain the lattice $V_{\N}$ is to find the lattice $V_{\N_{\E_j}}$ for one of the subnetworks $\NN_{\E_j}$, with $j \in \{1,\ldots,l\}$, and then find the subset of subspaces in $V_{\N_{\EE_j}}$ that are left invariant by the adjacency matrices of the other subnetworks $\NN_{\E_k}$,  $k \in \{1,\ldots,l\} \setminus \{j\}$. 

\begin{ex}  \normalfont 
Consider the homogeneous network $\NN$  with asymmetric inputs in Figure~\ref{fig:7cnetwork} and its identical-edge subnetworks $\NN_{\E_1}$ and $\NN_{\E_2}$, in Figures~\ref{fig:subN1} and \ref{fig:subN2}, with $\EE_1$ and $\EE_2$ the solid and dashed edge-types, respectively. Tables~\ref{t:rec_ex_nb}, \ref{t:rec_ex_pb} and \ref{t:rec_ex_npb} describe the synchrony subspaces in the lattice $V_{\N_{\EE_1}}$ for the subnetwork $\NN_{\E_1}$ obtained in Example~\ref{ex:lattice_union} of Section~\ref{sec:lattice_union}.  The subnetwork $\NN_{\E_1}$ is the union of two $1$-input regular coupled cell networks. In Section~\ref{sec:lattice_union} we see how to obtain the lattice of synchrony subspaces for the union of two networks from their lattice of synchrony subspaces.  The lattice $V_{\N_{\EE_2}}$ of the synchrony subspaces for $\NN_{\E_2}$ is given in Tables~\ref{tab:irredu} and \ref{tab:nonirredu}. This lattice is determined in Example~\ref{ex:lattice_1_regular} using the results in Section~\ref{sec:lattice_1_regular} for the lattice of synchrony subspaces of $1$-input regular coupled cell networks. 
By Corollary~\ref{cor:hom}, making the intersection of the lattices $V_{\N_{\EE_1}}$ and $V_{\N_{\EE_2}}$, the lattice $V_{\N}$ for the network $\NN$ is given by the synchrony subspaces in Table~\ref{t:lattice_hom_asym}.

\begin{table}[!htb]   
\begin{center}
\scalebox{.8}
{
\begin{tabular}{|l|l|}
\hline
\multicolumn{1}{|c}{\Gape[.2cm][.2cm]{\hspace{2.3in}$V_{\N}$}} & \\
\hline 
\hline
\Gape[.1cm][.1cm]
{$\{ {\bf x} :\  x_3 = x_7 \}$} & $\{ {\bf x} :\  x_1=x_2 =x_3=x_4=x_5=x_7\}$ \\
$\{ {\bf x} :\  x_1=x_3=x_7,\ x_2 = x_4=x_5\} $ & $ \{ {\bf x} :\  x_1=x_2 =x_3=x_4=x_5=x_6=x_7 \}$ \\
$\{ {\bf x} :\  x_1=x_3=x_6=x_7,\ x_2 = x_4=x_5\} $ &  \\
\hline
\end{tabular}
}
\caption{The lattice $V_{\N}$ of the synchrony subspaces for the network $\NN$ in Figure~\ref{fig:7cnetwork}.} \label{t:lattice_hom_asym}
\end{center}
\end{table}

\END
\end{ex}

\section{Lattice of synchrony subspaces for the union of networks} \label{sec:lattice_union}

As  already mentioned, each identical-edge subnetwork of a given coupled cell network may be a (connected) network or the union of (connected) networks with the same cell and edge-types. In this section we show how to obtain the lattice of synchrony subspaces for the disjoint union of two coupled cell networks with the same cell and edge-types from the lattice of synchrony subspaces of those networks.

We start by introducing the definitions of non-bipartite, pairing bipartite and non-pairing bipartite synchrony subspace. This terminology was first introduced in Aguiar and Ruan~\cite{AR12} in terms of balanced equivalence relations. The definitions presented here are equivalent to those for balanced equivalence relations in \cite{AR12}.

\begin{Def}\rm \label{Def:bip}\label{Def:pairbip} 
Let $\NN=\NN_1 + \NN_2$ and $\Delta_{\bowtie} \in V_{\N}$  a synchrony subspace for $\NN$. The  synchrony subspace $\Delta_{\bowtie}$ is called {\it bipartite}, if there exists at least one coordinate equality condition $x_{c_1} = \cdots = x_{c_m}$ in the definition of the polydiagonal $\Delta_{\bowtie}$ such that $\{ c_1, \ldots, c_m\} \cap  \CC_1 \ne \emptyset$ and $\{ c_1, \ldots, c_m\} \cap  \CC_2 \ne \emptyset$. Otherwise, $\Delta_{\bowtie}$ is called {\it non-bipartite}. If $\Delta_{\bowtie}$ is {\it bipartite} then it is called {\it pairing}, if every equality condition in $\Delta_{\bowtie}$  is of the form $x_{c_1} = x_{c_2}$ with $c_1  \in  \CC_1$ and $c_2 \in \CC_2$ and there is no cell $c_3 \in \left(C_1\setminus \{c_1\} \cup C_2 \setminus \{c_2\} \right)$ such that $x_{c_1} =x_{c_2} = x_{c_3}$.
\END\end{Def}

Let $\NN_i = \left(\CC_i, \EE_i, \sim_{\CC_i} = \{\CC_i\},\sim_{\EE_i} = \{\EE_i\}\right)$, $i=1,2$, be two regular coupled cell networks with the same cell and edge-type and such that $\CC_1 \cap \CC_2 = \emptyset$ and $\EE_1 \cap \EE_2 = \emptyset$, and  $\NN = \NN_1 +  \NN_2$ the disjoint union of the networks $\NN_1$ and $\NN_2$. 
We denote by $V^{nb}_{\N}$, $V^{pb}_{\N}$ and $V^{npb}_{\N}$ the sets
 of the non-bipartite, pairing bipartite and non-pairing bipartite synchrony subspaces for $\NN$, respectively.

Given a synchrony subspace  $\Delta_{\bowtie} \in V_{\N}$, for $i=1,2$, we define the polydiagonal
$$
\Delta_{\bowtie_{P_i}} =\{ x \in \left(\RR^k\right)^{n_i}:\ x_c = x_d \mbox{ for all } c, d \in C_i \mbox{ such that } x_c = x_d \mbox{ is a condition in }  \Delta_{\bowtie} \},
$$
with $n_i = \# \CC_i$. 

Analogously, for $i=1,2$, given a synchrony subspace  $\Delta_{\bowtie_i} \in V_{\N_i}$ we define the polydiagonal 
$$
\Delta_{\bowtie_{E_i}} =\{ x \in \left(\RR^k\right)^n:\ x_c = x_d  \mbox{ such that } x_c = x_d \mbox{ is a condition in }  \Delta_{\bowtie_i} \},
$$
with $n=n_1+n_2$. 

We note that,  as the cells in $\CC_1$ do not receive inputs from the cells in $\CC_2$, and vice-versa, the polydiagonal $\Delta_{\bowtie_{P_i}}$, $i=1,2$, is a synchrony subspace for $\NN_i$ and the polydiagonal $\Delta_{\bowtie_{E_i}}$ is a synchrony subspace for $\NN$. 
 
Given two synchrony subspaces $\Delta_{\bowtie_1} \in V_{\N_1}$ and $\Delta_{\bowtie_2} \in V_{\N_2}$ we denote by $\Delta_{\bowtie_1}  \dot \cap  \Delta_{\bowtie_2}$ the intersection $\Delta_{\bowtie_{E_1}}  \cap\Delta_{\bowtie_{E_2}}$. 
Remember Definition~\ref{def:intsym} of interior symmetry.

The next theorem states that the results in Section 4.4.1 of Aguiar and Ruan~\cite{AR12} for the join of networks, in terms of balanced equivalence relations, are also valid for the union of networks, in terms of synchrony subspaces. For more details, see Remark~\ref{rmk:join} below.

\begin{thm}\label{thm:main_join_same}
Let $\NN= \NN_1 + \NN_2$ be the disjoint union of two identical-edge networks $\NN_1,\NN_2$ with the same edge-type and valency $v_1$ and $v_2$, respectively. 

\begin{itemize}
\item[(a)] If $v_1 \ne v_2$ then $V_{\N} = V^{nb}_{\N}$. 
\item[(b)] If $v_1 = v_2$ then $V_{\N} = V^{nb}_{\N} \cup V^{pb}_{\N} \cup V^{npb}_{\N}.$
\end{itemize}

Moreover,
\begin{enumerate}
\item[(i)] $\Delta_{\bowtie} \in V^{nb}_{\N}$ if and only if $\Delta_{\bowtie}  = \Delta_{\bowtie_1}  \dot \cap  \Delta_{\bowtie_2}$ for $\Delta_{\bowtie_i}  \in  V_{\N_i}$, $i=1,2$;

\item[(ii)] $\Delta_{\bowtie} \in V^{pb}_{\N}$ with $\Delta_{\bowtie} = \{ {\bf x} : \ x_{c_i} = x_{d_i}, \ c_i \in \CC_1,\ d_i \in \CC_2, \ i=1,\dots,m \}$ if and only if the permutation given by the product of disjoint transpositions $(c_i,d_i)$, $i=1,\dots,m $ is an interior symmetry of $\NN$ on $\{c_1, d_1, \ldots, c_m,d_m\}$;

\item[(iii)] $\Delta_{\bowtie} \in V^{npb}_{\N}$ if and only if $\Delta_{\bowtie} = L\left( \Delta_{\bowtie_Q} \right)$, where $\Delta_{\bowtie_Q} \in V_Q^{pb}$, with $Q$ a quotient network associated to a nontrivial $\Delta_{\tilde \bowtie} \in V^{nb}_{\N}$ such that $\Delta_{\bowtie} \subset \Delta_{\tilde \bowtie}$.

\end{enumerate} 
\end{thm} 
\begin{proof}
This proof follows very closely that of Theorem 4.17 in Aguiar and Ruan~\cite{AR12} for the more general definition of the join of two networks designated by $f$-join.

We first note that, a necessary condition for a cell in $\NN_1$ and a cell in $\NN_2$ to synchronize is that they have the same valency. (a) Then, if $v_1 \ne v_2$ each cell in $\NN_i$, $i=1,2$ can synchronize only with cells in $\NN_i$ and, thus, there can be no bipartite synchrony subspaces for $\NN$. (b) If $v_1 = v_2$, the cells in $\NN_1$ and $\NN_2$ can synchronize and thus, by Definition~\ref{Def:bip}, a synchrony subspace either is non-bipartite, pairing bipartite or non-pairing bipartite. 

(i)  Let $\Delta_{\bowtie}$ be a non-bipartite synchrony subspace for $\NN$ and consider the associated synchrony subspaces $\Delta_{\bowtie_{P_1}}$ and  $\Delta_{\bowtie_{P_2}}$, for $\NN_1$ and $\NN_2$, respectively. Since $\Delta_{\bowtie}$ is a non-bipartite synchrony subspace, for every coordinate equality condition $x_{c_1} = \cdots = x_{c_m}$ in its definition, either $\{ c_1, \ldots, c_m\} \subseteq  \CC_1$ or $\{ c_1, \ldots, c_m\} \subseteq  \CC_2$. Thus, $\Delta_{\bowtie} = \left( \Delta_{\bowtie_{P_1}} \right)_{E_1} \cap \left( \Delta_{\bowtie_{P_2}} \right)_{E_2} =  \Delta_{\bowtie_{P_1}} \dot \cap \Delta_{\bowtie_{P_1}} $ 
On the contrary, consider that $\Delta_{\bowtie}  = \Delta_{\bowtie_1}  \dot \cap  \Delta_{\bowtie_2}$ for some $\Delta_{\bowtie_i}  \in  V_{\N_i}$, $i=1,2$, that is, $\Delta_{\bowtie}  =  \Delta_{\bowtie_{E_1}}  \cap  \Delta_{\bowtie_{E_2}}$. Then, in the definition of $\Delta_{\bowtie}$ there is no coordinate equality condition involving a cell in $\CC_1$ and a cell in $\CC_2$, that is, $\Delta_{\bowtie}$ is non-bipartite.
\vspace{0.1in}

(ii) Let $\Delta_{\bowtie}$ be a pairing bipartite synchrony subspace for $\NN$ and $x_{c_i} = x_{d_i}$, with $c_i \in \CC_1$ and  $d_i \in \CC_2$, for $i=1,\dots,m$, the coordinate equality conditions that define $\Delta_{\bowtie}$. For convenience, index the cells of $\NN$ by $b_1,\dots, b_n$ such that $c_i=b_{2i-1}$, $d_i=b_{2i}$ for $i=1,\dots, m$. Define $S=\{b_1,b_2,\dots,b_{2m-1},b_{2m}\}$ and $\sigma=(b_1\ b_2)(b_3\ b_4)\cdots(b_{2m-1}\ b_{2m})$. 
Let  $\mm(b_j,I(b_i))$ denote the number of times that $x_{b_j}$ appears as an input variable of the function $f$ defining the equation for $\dot x_{b_i}$ and $A:=(a_{ij})_{n \times n}$ be the adjacency matrix of $\NN$. Then,
\[a_{ij}=\mm(b_j,I(b_i)),\quad \forall\, 1\le i,j\le n.\] 
Consider $c_i=b_{2i-1}$, $d_i=b_{2i}$ and $c_j=b_{2j-1}$, $d_j=b_{2j}$ for some $i,j\in \{1,\dots,m\}$.  We have  that, $\Delta_{\bowtie}$ is a flow-invariant subspace satisfying the condition $x_{c_i} = x_{d_i}$ if and only if the equations for $\dot x_{c_i}$ and $\dot x_{d_i}$ coincide. Since, for $k>2m$, we have
\[a_{2i-1,k}=\mm(b_k,I(c_i))=\mm(b_k,I(d_i))=a_{2i,k},\]
that happens,  if and only if
{\small
\begin{equation}\label{eq:thm_pf_1}
a_{2i-1,2j-1}+a_{2i-1,2j}=\mm(c_j,I(c_i))+\mm(d_j,I(c_i))= \mm(c_j,I(d_i))+\mm(d_j,I(d_i)) = a_{2i,2j-1}+a_{2i,2j}.
\end{equation}
}
As the cell $c_i$ does not receive any input from the cell $d_j$ and the cell $d_i$ does not receive any input from the cell $c_j$, that is, $a_{2i-1,2j} = a_{2i,2j-1} =0$, the equality in (\ref{eq:thm_pf_1}) is equivalent to
\[a_{2i-1,2j-1}=a_{2i,2j}.\]

Therefore, we conclude that $\Delta_{\bowtie}$ is a pairing bipartite synchrony subspace if and only if $\sigma$ is an interior symmetry of $\NN$ on $S$.

(iii) Let $\Delta_{\bowtie} \in V^{npb}_{\N}$ be a non-pairing bipartite synchrony subspace for $\NN$. For simplicity of the proof, and without loss of generality, we are going to assume that $\Delta_{\bowtie}$ is defined only by two equality conditions $x_{c_1} = \cdots = x_{c_k} = x_{d_1} = \cdots = x_{d_l}$ and $x_{c_{k+1}} = \cdots = x_{c_{k+r}} = x_{d_{l+1}} = \cdots = x_{d_{l+s}}$, with $\{ c_1, \ldots, c_{k+r} \} \subseteq \CC_1$ and $\{ d_1, \ldots, d_{l+s}\} \subseteq \CC_2$.
 Since the cells in $\CC_1$ do not receive inputs from the cells in $\CC_2$, and vice-versa, we conclude that both the polydiagonal $\Delta_{\bowtie_{P_1}}$ defined by the equalities $x_{c_1} = \cdots = x_{c_k}$ and $x_{c_{k+1}} = \cdots = x_{c_{k+r}}$ and the polydiagonal $\Delta_{\bowtie_{P_2}}$  defined by the equalities $x_{d_1} = \cdots = x_{d_l}$ and $x_{d_{l+1}} = \cdots = x_{d_{l+s}}$ are synchrony subspaces for $\NN$. 
Consider the quotient network $Q_1$ associated to the synchrony subspaces $\Delta_{\bowtie_{P_1}}$  for $\NN_1$, where we identify cells $c_2,\ldots,c_k$ with cell $c_1$ and cells $c_{k+2},\ldots,c_{k+r}$ with cell $c_{k+1}$, and the quotient network $Q_2$ associated to the synchrony subspaces $\Delta_{\bowtie_{P_2}}$  for $\NN_2$, where we identify cells $d_2,\ldots,d_l$ with cell $d_1$ and cells $d_{l+2},\ldots,d_{l+s}$ with cell $c_{l+1}$
The network $Q = Q_1 + Q_2$ given by the disjoint union of the networks $Q_1$ and $Q_2$ corresponds to the quotient network of $\NN$ associated to the synchrony subspace $\Delta_{\tilde\bowtie} = \Delta_{\bowtie_1}  \dot \cap  \Delta_{\bowtie_2}$ in $V^{nb}_{\N}$. By Proposition~\ref{prop:rest_lift} (a), the restriction $\Delta_{\bowtie_Q} = R\left(  \Delta_{\bowtie}\right)$ of $\Delta_{\bowtie} \subset \Delta_{\bowtie_1}  \dot \cap  \Delta_{\bowtie_2} $ to $Q$ is a synchrony subspace for $Q$. Moreover, $\Delta_{\bowtie_Q}$ is a pairing bipartite synchrony subspace, as it is defined by the two equality conditions $x_{c_1} = x_{d_1}$ and $x_{c_{k+1}} = x_{d_{l+1}}$, and $\Delta_{\bowtie} = L\left( \Delta_{\bowtie_Q} \right)$.

On the other hand, let $\Delta_{\bowtie_Q}$ be a pairing bipartite synchrony subspace for the quotient network $Q$ of $\NN$ associated with a synchrony subspace $\Delta_{\tilde\bowtie}$ of $\NN$. Let $\Delta_{\bowtie} = L\left(\Delta_{\bowtie_Q} \right)$ be the lift of $\Delta_{\bowtie_Q}$ to $\NN$. By Proposition~\ref{prop:rest_lift} (b), $\Delta_{\bowtie}$ is a synchrony subspace for $\NN$. Moreover, since $\Delta_{\tilde\bowtie}$ is a nontrivial non-bipartite synchrony subspace for $\NN$ and $\Delta_{\bowtie_Q}$ is a pairing bipartite synchrony subspace for $Q$, we conclude that $\Delta_{\bowtie}$ is a nonparing bipartite synchrony subspace for $\NN$ and that $\Delta_{\bowtie} \subset \Delta_{\tilde\bowtie}$.
\qed
\end{proof}
\vspace{0.1in}

\begin{rmk}\rm
Let $\MM(V_{\N_i})$ be the set of the meet-irreducible elements for the lattice of synchrony subspaces of $\NN_i$, $i=1,2$, and define the set
$
\MM(V^{nb}_{\N} ) =  \left( \MM(V_{\N_1})  \dot \cap P_2 \right) \cup  \left( P_1  \dot \cap  \MM(V_{\N_2}) \right). 
$
From  Theorem~\ref{thm:main_join_same} (i) it follows that the elements in $V^{nb}_{\NN}$ are obtained by intersections of the synchrony subspaces in $\MM(V^{nb}_{\N} )$. Remember that the meet operation is the intersection.  
From Theorem~\ref{thm:main_join_same} (a) it follows that, if $v_1 \ne v_2$ then $\MM(V^{nb}_{\N} )$ is the set of the meet-irreducible elements for $V_{\N}$. 
\END 
\end{rmk}

\begin{rmk}\rm \label{rmk:union-1-input}
For the particular case of the disjoint union $\NN= \NN_1 + \NN_2$  of two 1-input regular networks $\NN_1$ and $\NN_2$, with the same edge-type, we have the following useful remarks:
\begin{itemize}
\item[(i)] Any interior symmetry of $\NN$ is on a set $S$ that has to include both the cells in the ring of $\NN_1$ and the cells in the ring of $\NN_2$.
\item[(ii)] Let $m_i$ be the number of cells in the ring of $\NN_i$, $i=1,2$. If $m_1 = m_2$ then there are interior symmetries of $\NN$ and, thus, there are pairing-bipartite synchrony subspaces for $\NN$. Otherwise, if $m_1 \ne m_2$ then there are no interior symmetries of $\NN$ and the set $V_{\N}^{pb}$ is empty.
\item[(iii)] Every quotient network $Q$ of $\NN$ associated to a non-bipartite synchrony subspace $\Delta_{\bowtie}$ is again the disjoint union of two 1-input regular networks $Q_1$ and $Q_2$, $Q = Q_1 + Q_2$, such that $Q_i$ is the quotient network of $\NN_i$ associated to $\Delta_{\bowtie_{P_i}}$, $i=1,2$. If $Q$ is a quotient network of $\NN$ associated to a bipartite (pairing or non-pairing) synchrony subspace then it is a 1-input regular network.
\item[(iv)] When using Theorem \ref{thm:main_join_same} (iii) to get the set $V_{\N}^{npb}$, we have to look for the interior symmetries of the quotient networks of $\NN$ associated to the non-bipartite synchrony subspaces. From (iii) it follows that (ii) is useful to identify the quotient networks that have or not interior symmetries.
\item[(v)] Although when $m_1 \ne m_2$ there are no interior symmetries for $\NN$, the same may not be true for the quotient networks of $\NN$. It may happen that the rings of the 1-input quotient networks $Q_1$ and $Q_2$ have the same number of cells. That happens in particular for the quotient network of the non-bipartite synchrony subspace where all cells of $\NN_1$ are synchronized and all cells of $\NN_2$ are synchronized. The quotient networks $Q_1$ and $Q_2$ will then consist of a unique cell with a self-loop and the network $Q$ has an interior symmetry on the set formed by those two cells. The lifting of the corresponding bipartite synchrony subspace is then the full synchronous subspace for $\NN$.
\end{itemize}
\END 
\end{rmk}

\begin{ex}  \normalfont \label{ex:lattice_union}
Consider the non connected network $\NN_{\E_1}$ in Figure~\ref{fig:subN1}, which in this example we will denote by $\NN$ to simplify the notation. We have $\NN = \NN_1 + \NN_2$, with $\NN_1$ and $\NN_2$ the networks on the left and right of Figure~\ref{fig:subN1}, respectively. We use Theorem~\ref{thm:main_join_same} to find the lattice $V_{\N}$ of the synchrony subspaces for $\NN$ from the lattices $V_{\N_1}$ and $V_{\N_2}$ of the synchrony subspaces for  $\NN_1$ and $\NN_2$, respectively. The lattices $V_{\N_1}$ and $V_{\N_2}$ are presented in Table \ref{t:rec_ex_G1G2}. Since $\NN_1$ and $\NN_2$ have the same valency, we have $V_{\N} = V^{nb}_{\N} \cup V^{pb}_{\N} \cup V^{npb}_{\N}$. By (i), we get the non-bipartite synchrony subspaces in $V^{nb}_{\N}$, which are listed in Table~\ref{t:rec_ex_nb}, by making the extended intersection of every synchrony subspace in $V_{\N_1}$ with every synchrony subspace in $V_{\N_2}$. The interior symmetries of $\NN$ given by the product of disjoint transpositions of the form $(c_i,d_i)$, with $c_i$ a cell in $\NN_1$ and $d_i$ a cell in $\NN_2$, are $(17)(25)$, $(12)(57)$, $(17)(25) (36)$ and $(12)(57)(34)$. By (ii), we get the pairing bipartite synchrony subspaces in $V^{pb}_{\N}$, which are presented in Table~\ref{t:rec_ex_pb}. By (iii), the non-pairing bipartite synchrony subspaces in $V^{npb}_{\N}$ are then obtained by considering, for every non-bipartite synchrony subspace, the corresponding quotient network and then making the lift of the bipartite synchrony subspaces for that quotient network, see Table~\ref{t:rec_ex_npb}. We note that, using Remark~\ref{rmk:union-1-input} (iv), we can conclude easily that for eighteen of the non-bipartite synchrony subspaces in $V^{nb}_{\N}$ the corresponding quotient network has no interior symmetries, and so, no pairing bipartite synchrony subspaces. Those eighteen  non-bipartite synchrony subspaces are: $\Delta_i$, $i \in \{ 1,4,5,6,8,9,11,12,16,19,22,23,24,26,27,29,30,34 \}$. Finally, the lattice $V_{\N}$ is formed by the 71 synchrony subspaces given in Tables~\ref{t:rec_ex_nb}-~\ref{t:rec_ex_npb}.
\END
\end{ex}

\begin{table}[!htb]   
\begin{center}
\scalebox{.8}
{\begin{tabular}{|l|l|}
\hline
\multicolumn{1}{|c|}{\Gape[.2cm][.2cm]{$V_{\N_1}$}}& \multicolumn{1}{c|}{$V_{\N_2}$}\\
\hline \hline
\Gape[.1cm][.1cm]{$P_1$}& $P_2$ \\
$\Delta^1_1 =\{ {\bf x} :\  x_2 = x_7\}$& $\Delta^2_1 =\{ {\bf x} :\  x_1 = x_5\}$\\
$\Delta^1_2 =\{ {\bf x} :\  x_3 = x_7\}$& $\Delta^2_2 =\{ {\bf x} :\  x_1 = x_6\}$\\
$\Delta^1_3 =\{ {\bf x} :\  x_2 = x_3 = x_7\}$& $\Delta^2_3 =\{ {\bf x} :\  x_4 = x_5\}$\\
& $\Delta^2_4 =\{ {\bf x} :\  x_1 = x_5 = x_6\}$\\
& $\Delta^2_5 =\{ {\bf x} :\  x_1 = x_4 = x_5\}$\\
& $\Delta^2_6 =\{ {\bf x} :\  x_1 = x_5,\  x_4= x_6\}$\\
& $\Delta^2_7 =\{ {\bf x} :\  x_1 = x_6,\  x_4= x_5\}$\\
& $\Delta^2_8 =\{ {\bf x} :\  x_1 = x_4 = x_5= x_6\}$\\
\hline
\end{tabular}
}
\caption{The lattices of synchrony subspaces for the networks $\NN_1$ and $\NN_2$ on the left and right of Figure~\ref{fig:subN1}, respectively.} \label{t:rec_ex_G1G2}
\end{center}
\end{table}

\begin{table}[!htb]   
\begin{center}
\scalebox{.8}
{
\begin{tabular}{|l|l|}
\hline
\multicolumn{1}{|c}{\Gape[.2cm][.2cm]{\hspace{3in}$V^{nb}_{\N}$}} & \\
\hline 
\hline
\Gape[.1cm][.1cm]
{$P_1 \times P_2$} & $\Delta_{18} = \Delta^1_2 \dot \cap P_2 = \{ {\bf x} :\  x_3 = x_7\}$ \\
$\Delta_1 = P_1  \dot \cap \Delta^2_1 =\{ {\bf x} :\  x_1 = x_5\}$ & $\Delta_{19} = \Delta^1_2 \dot \cap \Delta^2_1 =\{ {\bf x} :\  \ x_1 = x_5, \ x_3 = x_7\}$ \\
$\Delta_2 = P_1  \dot \cap \Delta^2_2 =\{ {\bf x} :\  x_1 = x_6\}$ & $\Delta_{20} = \Delta^1_2 \dot \cap \Delta^2_2 =\{ {\bf x} :\  x_1 = x_6,\ x_3 = x_7\}$ \\
$\Delta_3 = P_1  \dot \cap \Delta^2_3 =\{ {\bf x} :\  x_4 = x_5\}$ & $\Delta_{21} = \Delta^1_2 \dot \cap \Delta^2_3 =\{ {\bf x} :\  x_3 = x_7,\ x_4 = x_5\}$ \\
$\Delta_4 = P_1  \dot \cap \Delta^2_4 =\{ {\bf x} :\  x_1 = x_5 = x_6\}$ & $\Delta_{22} = \Delta^1_2 \dot \cap \Delta^2_4 =\{ {\bf x} :\  x_1 = x_5=x_6,\ x_3 = x_7\}$ \\
$\Delta_5 = P_1  \dot \cap \Delta^2_5 =\{ {\bf x} :\  x_1 = x_4 = x_5\}$ & $\Delta_{23} = \Delta^1_2 \dot \cap \Delta^2_5 =\{ {\bf x} :\  x_1 = x_4=x_5, \ x_3 = x_7\}$ \\
$\Delta_6 = P_1  \dot \cap \Delta^2_6 =\{ {\bf x} :\  x_1 = x_5,\  x_4= x_6\}$ & $\Delta_{24} = \Delta^1_2 \dot \cap \Delta^2_6 =\{ {\bf x} :\  x_1 = x_5, \ x_3 = x_7,\  x_4= x_6\}$ \\
$\Delta_7 = P_1  \dot \cap \Delta^2_7 =\{ {\bf x} :\  x_1 = x_6,\  x_4= x_5\}$ & $\Delta_{25} = \Delta^1_2 \dot \cap \Delta^2_7 =\{ {\bf x} :\  x_1 = x_6,\ x_3 = x_7,\  x_4= x_5\}$ \\
$\Delta_8 = P_1  \dot \cap \Delta^2_8 =\{ {\bf x} :\  x_1 = x_4 = x_5= x_6\}$ & $\Delta_{26} = \Delta^1_2 \dot \cap \Delta^2_8 =\{ {\bf x} :\   x_1 = x_4 = x_5= x_6,\ x_3 = x_7\}$ \\
$\Delta_9 = \Delta^1_1 \dot \cap P_2 = \{ {\bf x} :\  x_2 = x_7\}$ & $\Delta_{27} = \Delta^1_3 \dot \cap P_2 = \{ {\bf x} :\  x_2 = x_3 = x_7\}$ \\
$\Delta_{10} = \Delta^1_1 \dot \cap \Delta^2_1 =\{ {\bf x} :\   x_1 = x_5,\ x_2 = x_7 \}$ & $\Delta_{28} = \Delta^1_3 \dot \cap \Delta^2_1 =\{ {\bf x} :\  x_1 = x_5,\ x_2 = x_3= x_7\}$ \\
$\Delta_{11} = \Delta^1_1 \dot \cap \Delta^2_2 =\{ {\bf x} :\  x_1 = x_6,\ x_2 = x_7\}$ & $\Delta_{29} = \Delta^1_3 \dot \cap \Delta^2_2 =\{ {\bf x} :\  x_1 = x_6,\ x_2 = x_3= x_7\}$ \\
$\Delta_{12} = \Delta^1_1 \dot \cap \Delta^2_3 =\{ {\bf x} :\  x_2 = x_7,\ x_4 = x_5\}$ & $\Delta_{30} = \Delta^1_3 \dot \cap \Delta^2_3 =\{ {\bf x} :\  x_2 = x_3= x_7,\ x_4 = x_5\}$ \\
$\Delta_{13} = \Delta^1_1 \dot \cap \Delta^2_4 =\{ {\bf x} :\  x_1 = x_5=x_6,\ x_2 = x_7 \}$ & $\Delta_{31} = \Delta^1_3 \dot \cap \Delta^2_4 =\{ {\bf x} :\  x_1 = x_5=x_6,\ x_2 = x_3= x_7\}$\\
$\Delta_{14} = \Delta^1_1 \dot \cap \Delta^2_5 =\{ {\bf x} :\  x_1 = x_4=x_5,\ x_2 = x_7 \}$ & $\Delta_{32} = \Delta^1_3 \dot \cap \Delta^2_5 =\{ {\bf x} :\  x_1 = x_4=x_5,\ x_2 = x_3= x_7 \}$ \\
$\Delta_{15} = \Delta^1_1 \dot \cap \Delta^2_6 =\{ {\bf x} :\  x_1 = x_5,\ x_2 = x_7,\  x_4= x_6\}$ & $\Delta_{33} = \Delta^1_3 \dot \cap \Delta^2_6 =\{ {\bf x} :\  x_1 = x_5,\ x_2 = x_3= x_7,\  x_4= x_6\}$\\
$\Delta_{16} = \Delta^1_1 \dot \cap \Delta^2_7 =\{ {\bf x} :\  x_1 = x_6,\ x_2 = x_7,\  x_4= x_5\}$ & $\Delta_{34} = \Delta^1_3 \dot \cap \Delta^2_7 =\{ {\bf x} :\  x_1 = x_6,\ x_2 = x_3= x_7,\  x_4= x_5\}$\\
$\Delta_{17} = \Delta^1_1 \dot \cap \Delta^2_8 =\{ {\bf x} :\  x_1 = x_4 = x_5= x_6,\ x_2 = x_7 \}$ & $\Delta_{35} = \Delta^1_3 \dot \cap \Delta^2_8 =\{ {\bf x} :\   x_1 = x_4 = x_5= x_6,\ x_2 = x_3= x_7 \}$\\
\hline
\end{tabular}
}
\caption{The set $V^{nb}_{\N}$ of the non-bipartite synchrony subspaces for the network in Figure~\ref{fig:subN1}.} \label{t:rec_ex_nb}
\end{center}
\end{table}

\begin{table}[!htb]   
\begin{center}
\scalebox{.8}
{
\begin{tabular}{|l|l|}
\hline
\multicolumn{1}{|c}{\Gape[.2cm][.2cm]{\hspace{1.7in}$V^{pb}_{\N}$}} & \\
\hline 
\hline
\Gape[.1cm][.1cm]
{$\Delta_{36} = \{ {\bf x} :\  x_1 = x_7,\ x_2=x_5\}$} & $\Delta_{38} = \{ {\bf x} :\  x_1 = x_7,\ x_2=x_5,\ x_3=x_6\}$ \\
$\Delta_{37} = \{ {\bf x} :\  x_1=x_2,\ x_5 = x_7\} $ & $\Delta_{39} = \{ {\bf x} :\  x_1=x_2,\ x_3=x_4,\ x_5 = x_7 \}$ \\
\hline
\end{tabular}
}
\caption{The set $V^{pb}_{\N}$ of the pairing bipartite synchrony subspaces for the network in Figure~\ref{fig:subN1}.} \label{t:rec_ex_pb}
\end{center}
\end{table}

\begin{table}[!htb]   
\begin{center}
\scalebox{.8}
{
\begin{tabular}{|l|l|}
\hline
\multicolumn{1}{|c}{\Gape[.2cm][.2cm]{\hspace{3.5in}$V^{npb}_{\N}$}} & \\
\hline 
\hline
\Gape[.1cm][.1cm]
{$\Delta_{40} = \Delta_1 \cap \Delta_{36} = \{ {\bf x} :\  x_1 = x_2 = x_5=x_7\}$} & $\Delta_{56} = \Delta_5 \cap \Delta_{39} = \{ {\bf x} :\  x_1 = x_2 =x_3=x_4 =x_5=x_7\}$ \\
$\Delta_{41} = \Delta_1 \cap \Delta_{38} = \{ {\bf x} :\  x_1 = x_2 = x_5=x_7,\ x_3=x_6\}$ & $\Delta_{57} = \Delta_6 \cap \Delta_{36} = \{ {\bf x} :\  x_1 = x_2 =x_5=x_7,\  x_4=x_6\}$ \\
$\Delta_{42} = \Delta_1 \cap \Delta_{39} = \{ {\bf x} :\  x_1 = x_2 = x_5=x_7,\ x_3=x_4\}$ & $\Delta_{58} = \Delta_6 \cap \Delta_{38} = \{ {\bf x} :\  x_1 = x_2 =x_5=x_7,\  x_3 =x_4=x_6\}$ \\
$\Delta_{43} = \Delta_2 \cap \Delta_{36} = \{ {\bf x} :\  x_1 = x_6 = x_7,\ x_2=x_5\}$ & $\Delta_{59} = \Delta_7 \cap \Delta_{36} = \{ {\bf x} :\  x_1 = x_6 =x_7,\  x_2 =x_4=x_5\}$ \\
$\Delta_{44} = \Delta_2 \cap \Delta_{37} = \{ {\bf x} :\  x_1= x_2=x_6,\ x_5 = x_7 \}$ & $\Delta_{60} = \Delta_7 \cap \Delta_{37} = \{ {\bf x} :\  x_1 =x_2=x_6,\ x_4 = x_5 =x_7 \}$ \\
$\Delta_{45} = \Delta_2 \cap \Delta_{38} = \{ {\bf x} :\  x_1 = x_3 =x_6=x_7,\  x_2= x_5\}$ & $\Delta_{61} = \Delta_7 \cap \Delta_{38} = \{ {\bf x} :\  x_1 = x_3 =x_6 = x_7,\  x_2 =x_4=x_5\}$ \\
$\Delta_{46} = \Delta_2 \cap \Delta_{39} = \{ {\bf x} :\  x_1 =x_2=x_6,\  x_3= x_4,\  x_5 = x_7\}$ & $\Delta_{62} = \Delta_7 \cap \Delta_{39} = \{ {\bf x} :\  x_1 =x_2=x_6,\ x_3 = x_4 =x_5 = x_7 \}$ \\
$\Delta_{47} = \Delta_3 \cap \Delta_{36} = \{ {\bf x} :\  x_1 = x_7,\ x_2 =x_4=x_5\}$ & $\Delta_{63} = \Delta_8 \cap \Delta_{36} = \{ {\bf x} :\  x_1 = x_2 =x_4 = x_5=x_6=x_7\}$ \\
$\Delta_{48} = \Delta_3 \cap \Delta_{37} = \{ {\bf x} :\  x_1=x_2,\ x_4 = x_5 = x_7\}$ & $\Delta_{64} = \Delta_8 \cap \Delta_{38} = \{ {\bf x} :\  x_1 = x_2 = x_3=x_4 = x_5=x_6=x_7\}$ \\
$\Delta_{49} = \Delta_3 \cap \Delta_{38} = \{ {\bf x} :\  x_1 = x_7,\ x_2 = x_4 = x_5,\  x_3=x_6\}$ & $\Delta_{65} = \Delta_{18} \cap \Delta_{36} = \{ {\bf x} :\  x_1  = x_3=x_7,\  x_2=x_5\}$ \\
$\Delta_{50} = \Delta_3 \cap \Delta_{39} = \{ {\bf x} :\ x_1=x_2,\ x_3 = x_4 = x_5 = x_7 \}$ & $\Delta_{66} = \Delta_{18} \cap \Delta_{37} = \{ {\bf x} :\  x_1=x_2,\ x_3  = x_5=x_7 \}$ \\
$\Delta_{51} = \Delta_4 \cap \Delta_{36} = \{ {\bf x} :\  x_1 = x_2 = x_5 = x_6=x_7\}$ & $\Delta_{67} = \Delta_{19} \cap \Delta_{36} = \{ {\bf x} :\  x_1  = x_2=x_3 = x_5=x_7\}$ \\
$\Delta_{52} = \Delta_4 \cap \Delta_{38} = \{ {\bf x} :\  x_1 = x_2 = x_3 = x_5=x_6 = x_7\}$ & $\Delta_{68} = \Delta_{20} \cap \Delta_{37} = \{ {\bf x} :\  x_1 = x_2=x_6,\ x_3  = x_5=x_7  \}$ \\
$\Delta_{53} = \Delta_4 \cap \Delta_{39} = \{ {\bf x} :\  x_1 = x_2 = x_5 = x_6=x_7,\ x_3  = x_4\}$ & $\Delta_{69} = \Delta_{21} \cap \Delta_{36} = \{ {\bf x} :\  x_1  = x_3=x_7,\  x_2 = x_4=x_5\}$ \\
$\Delta_{54} = \Delta_5 \cap \Delta_{36} = \{ {\bf x} :\  x_1 = x_2 = x_4 = x_5=x_7\}$ & $\Delta_{70} = \Delta_{24} \cap \Delta_{36} = \{ {\bf x} :\  x_1 =x_2  = x_3=x_5 = x_7,\   x_4=x_6\}$ \\
$\Delta_{55} = \Delta_5 \cap \Delta_{38} = \{ {\bf x} :\  x_1 = x_2 = x_4 = x_5=x_7,\ x_3 = x_6\}$ & \\
\hline
\end{tabular}
}
\caption{The set $V^{npb}_{\N}$ of the non-pairing bipartite synchrony subspaces for the network in Figure~\ref{fig:subN1}.} \label{t:rec_ex_npb}
\end{center}
\end{table}

\begin{rmk}\rm
If $\Delta_{\bowtie_{nb}} \in V^{nb}_{\NN} \setminus \{ P_1 \times P_2\}$ and $\Delta_{\bowtie_{pb}} \in V^{pb}_{\NN}$ then, since the meet operation in the lattice of synchrony subspaces is the intersection, we have that $\Delta_{\bowtie}  = \Delta_{\bowtie_{nb}}  \cap  \Delta_{\bowtie_{pb}}$ is a synchrony subspace for $\NN$. Clearly, $\Delta_{\bowtie}$ is non-pairing bipartite. We note however that, the set $V^{npb}_{\NN}$ of the non-pairing bipartite synchrony subspaces may contain more subspaces besides these, as the reverse implication is not always true.
For the network in Example~\ref{ex:lattice_union}, it happens that every non-pairing bipartite synchrony subspace is given by the intersection of a non-bipartite synchrony subspace with a pairing bipartite synchrony subspace, check in Table~\ref{t:rec_ex_npb}. We present next an example where that is not true.
\END 
\end{rmk}

\begin{ex}  \normalfont \label{ex:Contra_ex}
Consider the network $\NN$ given by the disjoint union of the two 1-input regular coupled cell networks in Figure~\ref{fig:Contra_ex}. The non-pairing bipartite synchrony subspace for $\NN$ defined by the coordinate equality conditions $x_1=x_6=x_7$, $x_2=x_5=x_8$ and $x_3 = x_9$ is not given by the intersection of a non-bipartite and a pairing bipartite synchrony subspaces for $\NN$.
\END
\end{ex}

\begin{figure}[ht!]
\begin{center}
\vspace{-4mm}
\hspace{-4mm}
	{\small \begin{tikzpicture}[->,>=stealth',shorten >=1pt,auto, node distance=1.5cm, thick,node/.style={circle,draw}]
	      \node[node]	(1) at (-75mm, 4cm)  [fill=white] {$7$};
			\node[node]	(2) at (-5cm, 4cm)  [fill=white] {$2$};
			\node[node]	(3) at (-5cm, 2cm)  [fill=white] {$3$};
			\node[node]	(4) at (-25mm, 4cm)  [fill=white] {$1$};			
			\node[node]	(5) at (0cm, 4cm)  [fill=white] {$4$};			
			\node[node]	(6) at (0cm, 2cm)  [fill=white] {$6$};
			\node[node]	(7) at (-25mm, 2cm)  [fill=white] {$5$};
			\node[node]	(8) at (25mm, 2cm)  [fill=white] {$8$};
			\node[node]	(9) at (5cm, 2cm)  [fill=white] {$9$};
	
                                                   \path
				(2) edge [out=160, in=20] node {} (1)
				(1) edge [out=340, in=200] node {} (2)
				(2) edge node {} (3)
				
				(4) edge  node {} (5)
				(4) edge [out=240, in=120] node {} (7)
				(7) edge [out=60, in=300] node {} (4)
								(7) edge  node {} (6)
				 (6) edge  node {} (8) 
				 (8) edge  node {} (9) 
				;
		\end{tikzpicture}}
		\caption{Disjoint union of two 1-input regular coupled cell networks considered in Example~\ref{ex:Contra_ex}.}
		\label{fig:Contra_ex}
		\end{center}
\end{figure}
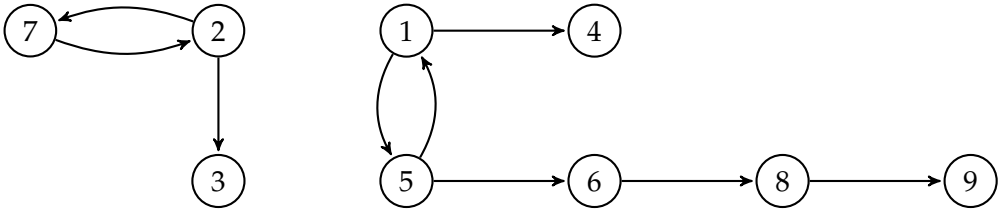

\begin{rmk}\label{rmk:join}\rm 
Consider two disjoint networks $\NN_1$ and $\NN_2$, that is, with no cell and no connection in common, with the same cell and edge-type. Following the usual definition of join of graphs, the join network $\NN = \NN_1 * \NN_2$ is given by the union of $\NN_1$ and $\NN_2$ together with additional edges from every cell of $\NN_1$ to every cell of $\NN_2$ and vice-versa. 
We have then that the results in Theorem~\ref{thm:main_join_same} $(i), (ii)$ and $(iii)$ for the union of two networks are also valid for the join. Moreover, if $n_i$ is the number of cells of $\NN_i$, $i=1,2$, in the join network  $\NN_1 * \NN_2$ the valency of the cells of $\NN_1$ is $v_1 +n_2$ and that of the cells of $\NN_2$ is $v_2 + n_1$. We have then for the join: (a) if $v_1 + n_2 \ne v_2 + n_1$ then $V_{\N} = V^{nb}_{\NN}$. (b) If $v_1 +n_2= v_2 + n_1$ then $V_{\N} = V^{nb}_{\NN} \cup V^{pb}_{\NN} \cup V^{npb}_{\NN}.$ 
(ii) Given two identical-edge networks $\NN_1$ and $\NN_2$ with the same edge-type the lattice of synchrony subspaces for the disjoint union $\NN_1 + \NN_2$ and the lattice of synchrony subspaces for the join $\NN_1 * \NN_2$ are different, in general.
 In the case of the networks $\NN_1$ and $\NN_2$ on the left and right of Figure~\ref{fig:subN1}, respectively, the lattice of synchrony subspaces for the disjoint union $\NN_1 + \NN_2$ is formed by the 71 synchrony subspaces given in the previous example. Nevertheless, the lattice of the synchrony subspaces for the join $\NN_1 * \NN_2$ is formed only by the non-bipartite synchrony subspaces in Table~\ref{t:rec_ex_nb}. Since network $\NN_1$ has 3 cells and network $\NN_2$ has 4 cells, in the join network  $\NN_1 * \NN_2$ the valency of the cells of $\NN_1$ is $5$ and that of the cells of $\NN_2$ is $4$. 
 It follows then that there are no bipartite synchrony subspaces for the join $\NN_1 * \NN_2$ .
(iii) Given two identical-edge networks $\NN_1$ and $\NN_2$ with the same edge-type the lattice of synchrony subspaces for the disjoint union $\NN_1 + \NN_2$ and the lattice of synchrony subspaces for the join $\NN_1 * \NN_2$ coincide if and only if $v_1 = v_2$ and $n_1 = n_2$ or $v_1 \ne v_2$ and $v_1 + n_2 \ne v_2 + n_1$.
 \END 
\end{rmk}

\section{Lattice of synchrony subspaces for 1-input regular coupled cell networks} \label{sec:lattice_1_regular} 

In this section we characterize the set of join-irreducible elements for the lattice of synchrony subspaces of a 1-input regular coupled cell network by describing its elements and indicating its cardinality. This set is a join-dense set for the lattice. Each join-irreducible synchrony subspace is identified by a basis of eigenvectors or generalized eigenvectors. We also describe the possible patterns of balanced colourings associated to the join-irreducible elements and conclude about the possible patterns of balanced colourings associated to the other synchrony subspaces in the lattice.

For an identical-edge coupled cell network $\NN$ with adjacency matrix $A$, assuming the cell phase spaces to be $\RR$, it follows from Remark~\ref{rmk:moreonlinear}  that a polydiagonal subspace is a synchrony subspace for $\NN$ if and only if it is left invariant by the matrix $A$. As in Example~\ref{ex:sublattice}, the set of the $A$-invariant subspaces is a lattice where the partial order is the inclusion and the meet and join operations are the intersection and sum, respectively. 

Although, the set of synchrony subspaces for $\NN$ is a lattice, and a subset of the set of the $A$-invariant subspaces, it is not in general a sublattice of the lattice of the $A$-invariant subspaces, as shown by Stewart~\cite{S07}. The meet operation is the same, the intersection of subspaces, but the join of two synchrony subspaces may not be given by their sum. The join of two synchrony subspaces is given by their sum only when this is a polydiagonal subspace. Note that the sum of two synchrony subspaces is always $A$-invariant but it may not be a polydiagonal subspace.  Apparently, there is no general form for the join operation in the lattice of synchrony subspaces. It follows then that, in general, it is not possible to define the join-irreducible set for the lattice of synchrony subspaces and obtain the lattice through that join-dense set.

The situation changes when one considers networks with asymmetric inputs. For this particular type of networks, as shown in Aguiar~\etal~\cite{AADF11}, the set of synchrony subspaces is closed under the sum operation.

\begin{thm}[\cite{AADF11} Theorem 3.8] \label{thm:AADF}
Let $\NN$ be a coupled cell network with asymmetric inputs and $\Delta_1$ and $\Delta_2$ two distinct synchrony subspaces for $\NN$. Then the sum $\Delta_1 + \Delta_2$ is a synchrony subspace for $\NN$.
\end{thm}
\begin{proof}
For completeness, we include here a proof that follows the same key ideas of the proof in \cite{AADF11} but that makes use of the definitions and concepts that are used in the scope of this work.

Let $\Delta_1$ and $\Delta_2 $ be two distinct synchrony subspaces for $\NN$. Consider the set  $E$ of the coordinate equality conditions that appear in the definition of both $\Delta_1$ and $\Delta_2$. If $E = \emptyset$ then $\Delta_1 + \Delta_2$ is the total phase space and there is nothing to prove. Otherwise, consider the polydiagonal subspace  $\Delta_{E}$ defined by the set $E$ of coordinate equality conditions. We have $\Delta_1 + \Delta_2 \subseteq \Delta_{E}$. Moreover, $\Delta_1 + \Delta_2$ is a polydiagonal, and thus, a synchrony subspace, if and only if $\Delta_1 + \Delta_2 = \Delta_{E}$.  This happens if and only if $\Delta_{E}$ is a synchrony subspace.

By Corollary~\ref{cor:ADlinear} and Remark~\ref{rmk:moreonlinear}, the polydiagonal $\Delta_{E}$ is a synchrony subspace for $\NN$ if and only if it is left invariant by each of the network adjacency matrices, one per each edge-type. Let $A$ be any of the adjacency matrices for $\NN$. To show that $\Delta_{E}$ is left invariant by $A$ we have to show that for any vector $u$ in $\Delta_{E}$ the vector $A u$ is in $\Delta_{E}$. Since the network is asymmetric, each row of $A$ has only one nonzero entry, which is equal to $1$. Thus, each coordinate of $A u$ is equal to a coordinate of $u$. 
For each coordinate equality condition that defines $\Delta_{E}$, either $Au$ satisfies trivially that condition or satisfies it if and only if $u$ satisfies a certain coordinate equality condition $C$. Since $\Delta_1$ and $\Delta_2 $ are synchrony subspaces for $\NN$, and thus invariant by the matrix $A$, we conclude that $C$ is a coordinate equality condition in the definition of $\Delta_1$ and $\Delta_2$, and thus, in the definition of $\Delta_{E}$. It follows then that every $u \in \Delta_{E}$ satisfies the condition $C$. 
These arguments prove that $\Delta_{E}$ is $A$-invariant by each adjacency matrix $A$ of $\NN$ and, thus, is a synchrony subspace. We conclude then that $\Delta_1 + \Delta_2 = \Delta_{E}$ and, thus, is a synchrony subspace.
\qed 
\end{proof}

\vspace{0.1in}

Remember that a sublattice of a lattice $L$ is a nonempty subset of $L$ that is a lattice with the same meet and join operations as $L$. 
It follows then that for a coupled cell network with asymmetric inputs and adjacency matrices $A_{1}, \ldots, A_k$, for $l=1,\ldots,k$ the different edge-types, the lattice of synchrony subspaces is a sublattice of the lattice of the subspaces that are invariant by the $k$ matrices $A_{1}, \ldots, A_k$, with the meet operation given by the intersection and the join operation given by the sum. In particular, for 1-input regular coupled cell networks we have the following corollary. 

\begin{cor} \label{cor:sublattice}
Let $\NN$ be a 1-input regular coupled cell network with adjacency matrix $A$. The lattice of synchrony subspaces for $\NN$ is a sublattice of the lattice of the $A$-invariant subspaces.
\end{cor}

We have then that, the join operation for the lattice of synchrony subspaces of a 1-input regular coupled cell network is well defined, and is given by the sum. Moreover, since we are considering finite networks, by Theorem~\ref{thm:join-dense}, the set of join-irreducible synchrony subspaces is join-dense for that lattice. In the next section we identify the set of join-irreducible elements for the lattice of synchrony subspaces of a 1-input regular coupled cell network .

\subsection{Join-irreducible set} 
Let $\NN$ be a 1-input regular coupled cell network and $A$ the corresponding adjacency matrix. By Corollary~\ref{cor:sublattice}, the lattice of synchrony subspaces for $\NN$ is a sublattice of the lattice of the $A$-invariant subspaces. Thus, the elements in the lattice $V_{\N}$ are the $A$-invariant subspaces that are defined by equalities of cell coordinates. From those, the join-irreducible elements are the ones that cannot be given by the sum of other elements in the lattice. 

The $A$-invariant subspaces are the subspaces generated by any number of eigenvectors, any number of Jordan chains or any number of eigenvectors and Jordan chains of $A$. We start then by characterizing the eigenvectors and Jordan chains of the adjacency matrix $A$ of a 1-input regular coupled cell network.

Let $\NN$ be an $n$-cell 1-input regular coupled cell network with a $m$-cell ring, with $1 \le m \le n$. The cells of $\NN$ can be enumerated such that the corresponding adjacency matrix $A$ has the following block form:
\begin{equation} \label{eq:matrix}
A =
\left[
\begin{array}{ll}
C_{m \times m} & 0_{m \times (n-m)} \\
B_{(n-m) \times m} & L_{(n-m) \times (n-m)}
\end{array}
\right],
\end{equation}
where $C$ is a cyclic permutation matrix with rows and columns indexed by the cells of the ring and $L$ is a lower triangular matrix. 

In what follows, we use the notation $e_j$, for $j=1,\ldots,n$, to represent the vector with the $j$-th entry  equal to $1$ and all the other entries equal to zero. 

\begin{lem} \label{lem:eigen}     
Let $\NN$ be a 1-input regular coupled cell network with $n$ cells $\{c_1, \ldots, c_n \}$ and a $m$-cell ring, with $1 \le m \le n$. Assume that the cells of $\NN$ are enumerated such that the corresponding adjacency matrix $A$ has the block structure (\ref{eq:matrix}). We have:

\begin{itemize}
\item[(i)] The set of eigenvalues of $A$ include the $m$ roots of the unity $\omega_j=e^{\frac{2\pi i j}{m}}$, $j=0,1,\dots, m-1$. The eigenspace associated to each eigenvalue $\omega_j$ is generated by the eigenvector $v_j=(1,\omega_j,\omega_j^2,\dots, \omega_j^{m-1}, v_{j,m+1}, \ldots, v_{j,n})$, where, for all $i \in \{m+1, \ldots, n \}$, we have $v_{j,i} = \omega_j^{k}$ for some $k \in \{0, \ldots, m-1\}$.

\item[(ii)] If $depth(\NN) =0$, that is, if $m = n$ there are no more eigenvalues for $A$. If $depth(\NN) \ne 0$, the other eigenvalue of $A$ is zero with algebraic multiplicity $n-m$.
The associated eigenspace has dimension $t$, with $t$ the number of tails of $\NN$, and a basis $(e_{i_1},\ldots,e_{i_t})$, with $i_j \in \{m+1, \ldots, n\}$, for $j \in {1,\ldots,t}$, such that $c_{i_j}$ is the leaf of the tail $j$.

\item[(iii)] If $depth(\NN) > 1$, the adjacency matrix $A$ is non semi-simple and the Jordan chains associated to zero can be obtained directly from the subtrees of the directed rooted trees of $\NN$. Let $S_{\T}$ be the set of the subtrees of the trees of $\NN$.  A chain of vectors $(\tilde u_s, \tilde u_{s-1}, \ldots, \tilde u_{1})$ is a Jordan chain associated to the eigenvalue zero if and only if there is a non-empty subset $J$ of $S_{\T}$, with $s= max \{ depth(\BB)|\ \BB \in J \} > 1$, such that, 
for each subtree $\BB_i \in J$, if we consider the sequence of vectors 
$\tilde u^{i}_1, \ldots, \tilde u^{i}_{s_{i}}, \ldots, \tilde u^{i}_{s}$, with $s_{i}=depth(\BB_i)$, there is $\alpha_i \in \RR$ such that $\tilde u^{i}_1 = \alpha_i e^{i}_{i^1_1}$, with $c^{i}_{i^1_1}$ the root cell of $\BB_i$, $\tilde u^{i}_j = \alpha_i \left( e^{i}_{i^j_{1}} + \cdots +  e^{i}_{i^j_{k}} \right)$, for $1< j \le s_{i}$, with  $c^{i}_{i^j_{l}}$, $1 \le l \le k$, the cells of the subtree $\BB_i$ such that $|c^{i}_{i^1_1}, c^{i}_{i^j_{l}}| = j-1$ and $u^{i}_{s_{i}+1}, \ldots, \tilde u^{i}_{s} =0$, then $ \tilde u_j =\sum_{\BB_i \in J} \tilde u^{i}_j$, $1 \le j \le s$.

 \end{itemize}
 \end{lem}

\begin{proof}
It follows from the lower triangular block form (\ref{eq:matrix}) that the eigenvalues of $A$ are given by the eigenvalues of $C$ and the eigenvalues of $L$.

The cyclic permutation matrix $C$ is a special kind of a circulant matrix that can be written as $circ(0,1,0,\dots, 0)$ for a shorthand. The eigenvalues of $C$ are the $m$ roots of the unity 
\begin{equation}\label{eq:lamj}
\omega_j=e^{\frac{2\pi i j}{m}},\,\, j=0,1,\dots, m-1.
\end{equation}
The lower triangular matrix $L$ has zeros at the diagonal and therefore only the zero eigenvalue. We have then that besides the $m$ roots of the unity, the other eigenvalue of $A$ is zero with algebraic multiplicity $n-m$.

The eigenvectors of $C$ associated to the eigenvalues $\omega_j$, for $ j=0,1,\dots, m-1$ are given by
\begin{equation}\label{eq:vj}
(1,\omega_j,\omega_j^2,\dots, \omega_j^{m-1})
\end{equation}
and thus, taking into account the entries of the submatrices $B$ and $L$, the corresponding eigenvectors for $A$ are given by
\begin{equation}\label{eq:vjMG}
v_j=(1,\omega_j,\omega_j^2,\dots, \omega_j^{m-1}, v_{j,m+1}, \ldots, v_{j,n}),
\end{equation}
where, for all $i \in \{m+1, \ldots, n \}$, we have $v_{j,i} = \omega_j^{k}$ for some $k \in \{0, \ldots, m-1\}$. This finishes the proof of (i).

The first $m$ columns of the matrix $A$, that correspond to the cells in the ring, are linearly independent, as the matrix $C$ has nonzero determinant. Let $t$ be the number of tails of $\NN$ and $c_{i_j}$, with $i_j \in \{m+1, \ldots, n\}$, for $j \in {1,\ldots,t}$, the leaf cells of the tails. The columns of $A$ corresponding to the leaf cells $c_{i_j}$, for $j \in {1,\ldots,t}$, have all entries equal to zero. The remaining $n-m-t$ columns of $A$, corresponding to the other cells in the tails, are linearly independent, as each of them has at least one nonzero entry (equal to $1$) and when one entry is nonzero for one of the columns it is zero for the other ones. Moreover, as they have the first $m$ coordinates equal to zero, they are linearly independent with the first $m$ columns of $A$. We have then that the eigenspace associated to zero has dimension $t$ and, for example, the basis $(e_{i_1},\ldots,e_{i_t})$. This finishes the proof of (ii).

Let $J$ be a non-empty subset of the set $S_{\T}$, of the subtrees of the trees of $\NN$, such that $s= max \{ depth(\BB)|\ \BB \in J \} > 1$. Let $C_1= \cup_{i=1}^{p_1}C^i_1$, with $C^i_1 = \{c^{i}_{i^1_1} \} \subset \{m+1, \ldots, n\}$, be the subset of the root cells of the subtrees in $J$, with $p_1 =\# J$.
For $2 \le j \le s$, let $C^{i}_{j} = \{ c^{i}_{i^j_{1}}, \ldots, c^{i}_{i^j_{k_{ij}}} \} \subset \{m+1, \ldots, n\}$ be the subset of cells of the subtree $\BB_i$ that receive an input from a cell in $C^{i}_{j-1}$, for $1 \le i \le p_1$. Thus, $C_j = \cup_{i=1}^{p_1}C^{i}_j$ is the subset of cells in the subtrees in $J$ that receive an input from a cell in $C_{j-1}$.
Then, $A \left(\sum_{i=1}^{p_1} \alpha_i \left(  e^{i}_{i^{j-1}_{1}} + \cdots + e^i_{i^{j-1}_{{k_{i,j-1}}}}\right) \right) = \sum_{i=1}^{p_1} \alpha_i \left (e^i_{i^j_1} + \cdots + e^i_{i^j_{k_{ij}}}\right)$, for $2 \le j \le s$, and $\alpha_i \in \RR$. Note that the cells in $C_s$ are leafs and, thus,  $A \left( \sum_{i=1}^{p_1} \alpha_i \left(e^i_{i^s_1} + \cdots + e^i_{i^s_{k_{is}}}\right) \right)= 0$. If we set $\tilde u_j  =  \sum_{i=1}^{p_1} \alpha_i \left(e^i_{i^j_1} + \cdots + e^i_{i^j_{k_{ij}}}\right)$, for $1 \le j \le s$ then, for every $\alpha_i \in \RR$, we have that $(\tilde u_s, \tilde u_{s-1}, \ldots, \tilde u_{1})$ is a Jordan chain associated to the eigenvalue zero.
Conversely, if $(\tilde u_s, \tilde u_{s-1}, \ldots, \tilde u_{1})$, with $1 < s \le m-n$ is a Jordan chain associated to the eigenvalue zero then $A( \tilde u_s) =  0$ and $A( \tilde u_{j-1}) =  \tilde u_j$, for $2 \le j \le s$. Since the network $\NN$ is 1-input regular, each row of $A$ has only one nonzero entry, which is equal to $1$. Thus, we have $\tilde u_j  = \alpha^j_{1} e_{k^j_1} + \cdots + \alpha^j_{p_j} e_{k^j_{p_j}}$, for $1 \le j \le s$, with $k^j_p \ne k^j_q$, for $p \ne q$.
We have then $C_{j} = \{ c_{k^j_1}, \ldots, c_{k^j_{p_j}} \} \subset \{m+1, \ldots, n\}$. Each cell in $C_j$ receives its input from a cell in $C_{j-1}$ and $C_l \cap C_p = \emptyset$, for all $1 \le l,p \le s$. A cell in $C_{j-1}$ that does not send an input to a cell in $C_{j}$ is a leaf cell of $\NN$. We have then that, the cells in $\cup_{j=2}^s C_{j}$ and their inputs, together with the cells in $C_1$ form the subset $J$ of subtrees of $\NN$ associated to the given Jordan chain. If cell $c_{k^j_l}$ receives its  input from cell $c_{k^{j-1}_{\tilde l}}$ then $\alpha_{k^j_l} = \alpha_{k^{j-1}_{\tilde l}}$. Thus, the cells in the same subtree $\BB_i$ are associated to the same $\alpha_i$ constant. Moreover, those  that have the same distance from the root cell of the subtree are in the same set $C_j$. As $s >1$, there is at least one subtree in $J$ with depth greater than 1.
This finishes the proof of (iii).
\qed 
\end{proof}

\begin{thm} \label{thm:setirredu}  
Let $\NN$ be a 1-input regular coupled cell network with $n$ cells $\{c_1, \ldots, c_n \}$ and a $m$-cell ring, with $1 \le m \le n$. Assume that the cells of $\NN$ are enumerated such that the corresponding adjacency matrix $A$ has the block structure (\ref{eq:matrix}) and consider the notation introduced in Lemma~\ref{lem:eigen}.
The set $\mathcal{J}(V_{\N})$ of the join-irreducible elements for the lattice $V_{\N}$ of the synchrony subspaces of $\NN$ is formed by 
\begin{itemize}
\item[(i)] the $1$-dimensional synchrony subspace $E_{\omega_0}$, 
\item[(ii)] for each proper divisor $d$ of $m$ different from $1$, the $\frac{m}{d}$-dimensional synchrony subspace $E_{\omega_0} + E_{\omega_d} + E_{\omega_{2d}} + \cdots + E_{\omega_{\left(\frac{m}{d}-1\right) d}}$, 
\item[(iii)] the $m$-dimensional synchrony subspace $E_{\omega_0} + E_{\omega_1}+ \cdots + E_{\omega_{m-1}}$,
\item[(iv)] for every $1 \le l \le t$ and every possible subset $\{ u_1, \ldots, u_l \} \subseteq \{ e_{i_1}, \ldots, e_{i_t} \} $, the $2$-dimensional synchrony subspace $E_{\omega_0} + <u_1 + \cdots + u_l>$, 
\item[(v)] for each Jordan chain $(\tilde u_1, \ldots, \tilde u_q)$, where each $\tilde u_j = e_{k^j_1} + \cdots + e_{k^j_{p_{j}}}$, the $(q+1)$-dimensional synchrony subspace $E_{\omega_0} + <\tilde u_1, \ldots, \tilde u_q>$.
\end{itemize}
\end{thm}

\begin{proof} 
By Corollary~\ref{cor:sublattice}, the synchrony subspaces for $\NN$ are the $A$-invariant subspaces that are defined by some set of cell coordinate equalities. The $A$-invariant subspaces are the subspaces generated by eigenvectors and/or Jordan chains of $A$. We start by determining, for each eigenvector and Jordan chain of $A$, the synchrony subspace with lowest dimension that contains it.
Then we argue that the synchrony subspaces $(i)-(v)$ form the set $\mathcal{J}(V_{\N})$ of the join-irreducible elements for the lattice $V_{\N}$ as they cannot be given by the sum of other synchrony subspaces and the other synchrony subspaces in the lattice are obtained by their sums.

(i) We first note that the eigenvector $v_0$ is $(1,1, \ldots, 1)$ and, thus, the eigenspace $E_{\omega_0}$ is the full synchrony subspace, where all cells are synchronized. Moreover, every synchrony subspace is given by the sum of  $E_{\omega_0}$ with other subspace(s) generated by eigenvectors and/or Jordan chains of $A$.

(ii) It follows from (i) in Lemma~\ref{lem:eigen} that for each eigenvalue $\omega_j=e^{\frac{2\pi i j}{m}}$, $j=1,\dots, m-1$, the associated eigenspace $E_{\omega_j}$ is generated by an eigenvector $v_j$ such that for all $i \in \{m+1, \ldots, n \}$, the $i$-th entry is equal to a $k$-th entry, for some $k \in \{1, \ldots, m\}$, where the pairs $(i,k)$ coincide for all $j$. 
Moreover, for each proper divisor $d$ of $m$ different from $1$, each eigenvector $v_j$, with $j \in \{ d, 2d, \ldots, \left(\frac{m}{d}-1\right) d \}$,  satisfies, for $i \in \{1,\ldots,\frac{m}{d}\}$ the coordinate equality conditions $v_{j,i} = v_{j,i +\frac{m}{d}} = \cdots = v_{j,i + \left(\frac{m}{d} -1\right)d}$. Therefore, $<v_0, v_d, v_{2d}, \cdots, v_{\left(\frac{m}{d}-1\right) d}>$ is a $\frac{m}{d}$-dimensional synchrony subspace defined by the cell coordinate equality conditions $x_{i} =  x_{i +\frac{m}{d}} = \cdots =  x_{i + \left(\frac{m}{d} -1\right)d}$, for $i \in \{1,\ldots,\frac{m}{d}\}$, together with the $n-m$ cell coordinate equality conditions $x_i=x_k$, with $i \in \{m+1, \ldots, n \}$ and $k \in \{0, \ldots, m-1\}$.

(iii) Moreover, note that the first $m$ entries of a vector $v_j$ such that $j \ne k d$, with $d$ a proper divisor of $m$ different from $1$, are all different. Thus, the synchrony subspace with lowest dimension that contains those vectors is the $m$-dimensional synchrony subspace $E_{\omega_0} + E_{\omega_1}+ \cdots + E_{\omega_{m-1}}$. 

(iv) It follows from (ii) in Lemma~\ref{lem:eigen} that each eigenvector associated to zero has the first $m$ entries equal to $0$ and the last $n-m$ entries equal to $0$ or $1$. Thus, the subspace generated by the eigenvector $\omega_0$ and any eigenvector associated to zero is a $2$-dimensional synchrony subspace.

(v) To each subset of subtrees of $\NN$ is associated an infinite number of Jordan chains, as proved in Lemma~\ref{lem:eigen} (iii). To each such subset is also associated an irreducible synchrony subspace as we will explain. For each subset of subtrees, let us start by considering the Jordan chain $(\tilde u_1, \ldots, \tilde u_q)$, such that $\tilde u_j = e_{k^j_1} + \cdots + e_{k^j_{p_{j}}}$, that is, such that the constants $\alpha_i$ are equal to $1$, with the constants $\alpha_i$ as defined in the proof of Lemma~\ref{lem:eigen} (iii). The situations where the constants $\alpha_i$ are all equal to another real constant value are equivalent to this.
We have then that each generalized eigenvector $\tilde u_i$, $i \in \{1,\ldots,q\}$ has the first $m$ entries equal to $0$ and the last $n-m$ entries equal to $0$ or $1$.  Thus, the coordinates of each generalized eigenvector $\tilde u_i$ can be separated into two sets, one with the coordinates that are equal to zero and the other with the coordinates that are equal to 1.
It follows also by Lemma~\ref{lem:eigen} (iii) that, if $\tilde u_{i,j} = 1$, for some $j \in \{m+1,\ldots,n\}$, then $\tilde u_{k,j} =0$ for all $k \in \{1,\ldots,q\} \setminus \{i\}$. Then, if $\tilde u_{i,j_1} = \cdots = \tilde u_{i,j_l}= 1$, for some $j_1, \ldots, j_l \in \{m+1,\ldots,n\}$, then $\tilde u_{k,j_1} = \cdots = \tilde u_{k,j_l}= 0$ for all $k \in \{1,\ldots,q\} \setminus \{i\}$ and thus $x_{j_1} = \cdots = x_{j_l}$ is a coordinate equality condition that is satisfied by all the generalized eigenvectors $\tilde u_i$, $i \in \{1,\ldots,q\}$.  Thus, the $n$ coordinates can be separated into $q+1$ sets such that for each of the $q$ generalized eigenvectors there is one set with the coordinates that are equal to $1$ in that eigenvector and equal to $0$ in the remaining generalized eigenvectors ; the remaining set contains the coordinates that are equal to zero in all the generalized eigenvectors.
We conclude then that $E_{\omega_0} + <\tilde u_1, \ldots, \tilde u_q>$ is a $(q+1)$-dimensional synchrony subspace. Moreover, it becomes clear, from the above, that each of the remaining Jordan chains associated to the subset of subtrees, where at least one of the constants $\alpha_i$ is different from the others, cannot generate a synchrony subspace. Each of those Jordan chains is contained in a synchrony subspace that is given by the join of synchrony subspaces. More concretely, for each different constant $\alpha_i$, consider the subset of the initial subset of subtrees with the subtrees related to that constant, and consider the synchrony subspace associated to each of those smaller subsets of subtrees. The lowest dimensional synchrony subspace that contains the Jordan chain is given by the join of those synchrony subspaces associated to the different constants $\alpha_i$.

The synchrony subspace $E_{\omega_0}$ is trivially irreducible. Each of the synchrony subspaces in $(ii)-(v)$ is irreducible because, among its generators, there is at least one eigenvector or Jordan chain that satisfies exactly the coordinate equality conditions that define the synchrony subspace.
Since each eigenvector and Jordan chain of $A$ is contained in at least one of the irreducible synchrony subspaces in $(i)-(v)$, or in a synchrony subspace that is given by the join of some of those irreducible synchrony subspaces, we conclude that those subspaces give the set $\mathcal{J}(V_{\N})$ of the join-irreducible elements of the lattice $V_{\N}$.
\qed 
\end{proof}

\begin{ex}  \normalfont \label{ex:lattice_1_regular}
Consider the $1$-input regular coupled cell network $\NN$ in Figure~\ref{fig:subN2}. The eigenvalues of the adjacency matrix of $\NN$ are $1, -1, i, -i$ and $0$. According to Lemma~\ref{lem:eigen}, the eigenspaces and generalized eigenspaces are
$$
E_1 = <(1,1,1,1,1,1,1)>,\quad E_{-1} = <(1,-1,1,-1,-1,1,1)>,\quad E_{i} = <(1,-i,-1,i,-i,-1,-1)>,
$$
$$
E_{-i} = <(1,i,-1,-i,i,-1,-1)>,\quad E_0 = <(0,0,0,0,0,1,0), (0,0,0,0,0,0,1)>,
$$
$$
G_0 = <(0,0,0,0,1,0,0), (0,0,0,0,0,1,0), (0,0,0,0,0,0,1)>.
$$

By Theorem~\ref{thm:setirredu} the join-irreducible elements in the set $\mathcal{J}(V_{\N})$ for the lattice $V_{\N}$ of the synchrony subspaces for $\NN$ are 
$$
\begin{array}{ll}
\Delta_{1} = E_1 &
\Delta_{2} = E_1 \oplus E_{-1} \\
\Delta_{3}  = E_1 \oplus E_{-1} \oplus E_{i} \oplus E_{-i} &
\Delta_{4} = E_1 \oplus <(0,0,0,0,0,0,1)>  \\
\Delta_{5}  = E_1 \oplus <(0,0,0,0,0,1,0)>  &
\Delta_{6} = E_1 \oplus <(0,0,0,0,0,1,1)> \\
\Delta_{7} =  E_1 \oplus <(0,0,0,0,1,0,0), (0,0,0,0,0,1,0)> &
\Delta_{8} =  E_1 \oplus <(0,0,0,0,1,0,1), (0,0,0,0,0,1,0)>.
\end{array}
$$
The balanced colourings corresponding to the  join-irreducible synchrony subspaces $\Delta_i$, $i=1,\ldots,8$ are presented in Table~\ref{tab:irredu}. The remaining synchrony subspaces in $V_{\N}$ are given by all the possible sums of the elements in $\mathcal{J}(V_{\N})$ and the corresponding balanced colourings are listed in Table~\ref{tab:nonirredu}. Note that, 
$\Delta_{2} + \Delta_{3} = \Delta_{3}$,  $\Delta_{4} + \Delta_{6} = \Delta_{4} + \Delta_{5}$,  $\Delta_{4} + \Delta_{8} =\Delta_{4} + \Delta_{7}$,
$\Delta_{5} + \Delta_{6} = \Delta_{4} + \Delta_{5}$, $\Delta_{5} + \Delta_{7} = \Delta_{7}$, $\Delta_{5} + \Delta_{8} =\Delta_{4} + \Delta_{7}$,
$\Delta_{6} + \Delta_{7} = \Delta_{6} + \Delta_{8} =\Delta_{4} + \Delta_{7}$ and $\Delta_{7} + \Delta_{8}=\Delta_{4} + \Delta_{7}$.

\begin{table}[!htb]   
\begin{center}
\scalebox{.55}
{
\begin{tabular}{|c|c|c|}
\hline
\\
$\Delta_1$
{\small \begin{tikzpicture}	[->,>=stealth',shorten >=1pt,auto, node distance=1.5cm, thick,node/.style={circle,draw}]
	                 \node[node]	(1) at (-75mm, 4cm)  [fill=green] {$7$};
			\node[node]	(2) at (-5cm, 4cm)  [fill=green] {$2$};
			\node[node]	(3) at (-5cm, 2cm)  [fill=green] {$3$};
			\node[node]	(4) at (-25mm, 4cm)  [fill=green] {$1$};			
			\node[node]	(5) at (0cm, 4cm)  [fill=green] {$5$};			
			\node[node]	(6) at (0cm, 2cm)  [fill=green] {$6$};
			\node[node]	(7) at (-25mm, 2cm)  [fill=green] {$4$};				

				\path
				 [dashed] 
				 (2) edge node {} (1) 
				(2) edge node {} (3)
				(3) edge node {} (7)
				(7) edge node {} (4)
				(4) edge node {} (2)
				(4) edge node {} (5)
				(5) edge node {} (6)
				;
		\end{tikzpicture}}
\hspace{0.2in}	
&  
$\Delta_2$
{\small \begin{tikzpicture}	[->,>=stealth',shorten >=1pt,auto, node distance=1.5cm, thick,node/.style={circle,draw}]
	                 \node[node]	(1) at (-75mm, 4cm)  [fill=green] {$7$};
			\node[node]	(2) at (-5cm, 4cm)  [fill=orange] {$2$};
			\node[node]	(3) at (-5cm, 2cm)  [fill=green] {$3$};
			\node[node]	(4) at (-25mm, 4cm)  [fill=green] {$1$};			
			\node[node]	(5) at (0cm, 4cm)  [fill=orange] {$5$};			
			\node[node]	(6) at (0cm, 2cm)  [fill=green] {$6$};
			\node[node]	(7) at (-25mm, 2cm)  [fill=orange] {$4$};				

				\path
				 [dashed] 
				 (2) edge node {} (1) 
				(2) edge node {} (3)
				(3) edge node {} (7)
				(7) edge node {} (4)
				(4) edge node {} (2)
				(4) edge node {} (5)
				(5) edge node {} (6)
				;
		\end{tikzpicture}}
\hspace{0.2in}
&
$\Delta_3$
{\small \begin{tikzpicture}	[->,>=stealth',shorten >=1pt,auto, node distance=1.5cm, thick,node/.style={circle,draw}]
	                 \node[node]	(1) at (-75mm, 4cm)  [fill=green] {$7$};
			\node[node]	(2) at (-5cm, 4cm)  [fill=orange] {$2$};
			\node[node]	(3) at (-5cm, 2cm)  [fill=green] {$3$};
			\node[node]	(4) at (-25mm, 4cm)  [fill=yellow] {$1$};			
			\node[node]	(5) at (0cm, 4cm)  [fill=orange] {$5$};			
			\node[node]	(6) at (0cm, 2cm)  [fill=green] {$6$};
			\node[node]	(7) at (-25mm, 2cm)  [fill=magenta] {$4$};				

				\path
				 [dashed] 
				 (2) edge node {} (1) 
				(2) edge node {} (3)
				(3) edge node {} (7)
				(7) edge node {} (4)
				(4) edge node {} (2)
				(4) edge node {} (5)
				(5) edge node {} (6)
				;
		\end{tikzpicture}}		
\hspace{0.2in} \\
  \\
\hline
  \\
$\Delta_4$
{\small \begin{tikzpicture}	[->,>=stealth',shorten >=1pt,auto, node distance=1.5cm, thick,node/.style={circle,draw}]
	                 \node[node]	(1) at (-75mm, 4cm)  [fill=orange] {$7$};
			\node[node]	(2) at (-5cm, 4cm)  [fill=green] {$2$};
			\node[node]	(3) at (-5cm, 2cm)  [fill=green] {$3$};
			\node[node]	(4) at (-25mm, 4cm)  [fill=green] {$1$};			
			\node[node]	(5) at (0cm, 4cm)  [fill=green] {$5$};			
			\node[node]	(6) at (0cm, 2cm)  [fill=green] {$6$};
			\node[node]	(7) at (-25mm, 2cm)  [fill=green] {$4$};				

				\path
				 [dashed] 
				 (2) edge node {} (1) 
				(2) edge node {} (3)
				(3) edge node {} (7)
				(7) edge node {} (4)
				(4) edge node {} (2)
				(4) edge node {} (5)
				(5) edge node {} (6)
				;
		\end{tikzpicture}}
 &  
 $\Delta_5$
{\small \begin{tikzpicture}	[->,>=stealth',shorten >=1pt,auto, node distance=1.5cm, thick,node/.style={circle,draw}]
	                 \node[node]	(1) at (-75mm, 4cm)  [fill=green] {$7$};
			\node[node]	(2) at (-5cm, 4cm)  [fill=green] {$2$};
			\node[node]	(3) at (-5cm, 2cm)  [fill=green] {$3$};
			\node[node]	(4) at (-25mm, 4cm)  [fill=green] {$1$};			
			\node[node]	(5) at (0cm, 4cm)  [fill=green] {$5$};			
			\node[node]	(6) at (0cm, 2cm)  [fill=orange] {$6$};
			\node[node]	(7) at (-25mm, 2cm)  [fill=green] {$4$};				

				\path
				 [dashed] 
				 (2) edge node {} (1) 
				(2) edge node {} (3)
				(3) edge node {} (7)
				(7) edge node {} (4)
				(4) edge node {} (2)
				(4) edge node {} (5)
				(5) edge node {} (6)
				;
		\end{tikzpicture}}
&
$\Delta_6$
{\small \begin{tikzpicture}	[->,>=stealth',shorten >=1pt,auto, node distance=1.5cm, thick,node/.style={circle,draw}]
	                 \node[node]	(1) at (-75mm, 4cm)  [fill=orange] {$7$};
			\node[node]	(2) at (-5cm, 4cm)  [fill=green] {$2$};
			\node[node]	(3) at (-5cm, 2cm)  [fill=green] {$3$};
			\node[node]	(4) at (-25mm, 4cm)  [fill=green] {$1$};			
			\node[node]	(5) at (0cm, 4cm)  [fill=green] {$5$};			
			\node[node]	(6) at (0cm, 2cm)  [fill=orange] {$6$};
			\node[node]	(7) at (-25mm, 2cm)  [fill=green] {$4$};					

				\path
				 [dashed] 
				 (2) edge node {} (1) 
				(2) edge node {} (3)
				(3) edge node {} (7)
				(7) edge node {} (4)
				(4) edge node {} (2)
				(4) edge node {} (5)
				(5) edge node {} (6)
				;
		\end{tikzpicture}}		
 \\
\\
\hline
  \\
$\Delta_7$
{\small \begin{tikzpicture}	[->,>=stealth',shorten >=1pt,auto, node distance=1.5cm, thick,node/.style={circle,draw}]
	                 \node[node]	(1) at (-75mm, 4cm)  [fill=green] {$7$};
			\node[node]	(2) at (-5cm, 4cm)  [fill=green] {$2$};
			\node[node]	(3) at (-5cm, 2cm)  [fill=green] {$3$};
			\node[node]	(4) at (-25mm, 4cm)  [fill=green] {$1$};			
			\node[node]	(5) at (0cm, 4cm)  [fill=orange] {$5$};			
			\node[node]	(6) at (0cm, 2cm)  [fill=yellow] {$6$};
			\node[node]	(7) at (-25mm, 2cm)  [fill=green] {$4$};					

				\path
				 [dashed] 
				 (2) edge node {} (1) 
				(2) edge node {} (3)
				(3) edge node {} (7)
				(7) edge node {} (4)
				(4) edge node {} (2)
				(4) edge node {} (5)
				(5) edge node {} (6)
				;
		\end{tikzpicture}}
 &  
 $\Delta_8$
{\small \begin{tikzpicture}	[->,>=stealth',shorten >=1pt,auto, node distance=1.5cm, thick,node/.style={circle,draw}]
	                 \node[node]	(1) at (-75mm, 4cm)  [fill=orange] {$7$};
			\node[node]	(2) at (-5cm, 4cm)  [fill=green] {$2$};
			\node[node]	(3) at (-5cm, 2cm)  [fill=green] {$3$};
			\node[node]	(4) at (-25mm, 4cm)  [fill=green] {$1$};			
			\node[node]	(5) at (0cm, 4cm)  [fill=orange] {$5$};			
			\node[node]	(6) at (0cm, 2cm)  [fill=yellow] {$6$};
			\node[node]	(7) at (-25mm, 2cm)  [fill=green] {$4$};				

				\path
				 [dashed] 
				 (2) edge node {} (1) 
				(2) edge node {} (3)
				(3) edge node {} (7)
				(7) edge node {} (4)
				(4) edge node {} (2)
				(4) edge node {} (5)
				(5) edge node {} (6)
				;
		\end{tikzpicture}}
		&		
  \\
\\
\hline
\end{tabular}
}
\caption{The balanced colourings corresponding to the  join-irreducible elements of the lattice $V_{\N}$ of the synchrony subspaces for the $1$-input regular coupled cell network $\NN$ in Figure~\ref{fig:subN2}.} \label{tab:irredu}
\end{center}
\end{table}

\begin{table}[!htb]   
\begin{center}
\scalebox{.55}
{
\begin{tabular}{|c|c|c|}
\hline
\\
$\Delta_2 + \Delta_4$
\hspace{-0.6in}
{\small \begin{tikzpicture}	[->,>=stealth',shorten >=1pt,auto, node distance=1.5cm, thick,node/.style={circle,draw}]
	                 \node[node]	(1) at (-75mm, 4cm)  [fill=yellow] {$7$};
			\node[node]	(2) at (-5cm, 4cm)  [fill=orange] {$2$};
			\node[node]	(3) at (-5cm, 2cm)  [fill=green] {$3$};
			\node[node]	(4) at (-25mm, 4cm)  [fill=green] {$1$};			
			\node[node]	(5) at (0cm, 4cm)  [fill=orange] {$5$};			
			\node[node]	(6) at (0cm, 2cm)  [fill=green] {$6$};
			\node[node]	(7) at (-25mm, 2cm)  [fill=orange] {$4$};				

				\path
				 [dashed] 
				 (2) edge node {} (1) 
				(2) edge node {} (3)
				(3) edge node {} (7)
				(7) edge node {} (4)
				(4) edge node {} (2)
				(4) edge node {} (5)
				(5) edge node {} (6)
				;
		\end{tikzpicture}}
\hspace{0.2in}	
&  
$\Delta_2 + \Delta_5$
\hspace{-0.6in}
{\small \begin{tikzpicture}	[->,>=stealth',shorten >=1pt,auto, node distance=1.5cm, thick,node/.style={circle,draw}]
	                 \node[node]	(1) at (-75mm, 4cm)  [fill=green] {$7$};
			\node[node]	(2) at (-5cm, 4cm)  [fill=orange] {$2$};
			\node[node]	(3) at (-5cm, 2cm)  [fill=green] {$3$};
			\node[node]	(4) at (-25mm, 4cm)  [fill=green] {$1$};			
			\node[node]	(5) at (0cm, 4cm)  [fill=orange] {$5$};			
			\node[node]	(6) at (0cm, 2cm)  [fill=yellow] {$6$};
			\node[node]	(7) at (-25mm, 2cm)  [fill=orange] {$4$};				

				\path
				 [dashed] 
				 (2) edge node {} (1) 
				(2) edge node {} (3)
				(3) edge node {} (7)
				(7) edge node {} (4)
				(4) edge node {} (2)
				(4) edge node {} (5)
				(5) edge node {} (6)
				;
		\end{tikzpicture}}
\hspace{0.2in}
&
$\Delta_2 + \Delta_6$
\hspace{-0.6in}
{\small \begin{tikzpicture}	[->,>=stealth',shorten >=1pt,auto, node distance=1.5cm, thick,node/.style={circle,draw}]
	                 \node[node]	(1) at (-75mm, 4cm)  [fill=yellow] {$7$};
			\node[node]	(2) at (-5cm, 4cm)  [fill=orange] {$2$};
			\node[node]	(3) at (-5cm, 2cm)  [fill=green] {$3$};
			\node[node]	(4) at (-25mm, 4cm)  [fill=green] {$1$};			
			\node[node]	(5) at (0cm, 4cm)  [fill=orange] {$5$};			
			\node[node]	(6) at (0cm, 2cm)  [fill=yellow] {$6$};
			\node[node]	(7) at (-25mm, 2cm)  [fill=orange] {$4$};			

				\path
				 [dashed] 
				 (2) edge node {} (1) 
				(2) edge node {} (3)
				(3) edge node {} (7)
				(7) edge node {} (4)
				(4) edge node {} (2)
				(4) edge node {} (5)
				(5) edge node {} (6)
				;
		\end{tikzpicture}}		
\hspace{0.2in} \\
  \\
\hline
  \\
$\Delta_2 + \Delta_7$
\hspace{-0.6in}
{\small \begin{tikzpicture}	[->,>=stealth',shorten >=1pt,auto, node distance=1.5cm, thick,node/.style={circle,draw}]
	                 \node[node]	(1) at (-75mm, 4cm)  [fill=green] {$7$};
			\node[node]	(2) at (-5cm, 4cm)  [fill=orange] {$2$};
			\node[node]	(3) at (-5cm, 2cm)  [fill=green] {$3$};
			\node[node]	(4) at (-25mm, 4cm)  [fill=green] {$1$};			
			\node[node]	(5) at (0cm, 4cm)  [fill=yellow] {$5$};			
			\node[node]	(6) at (0cm, 2cm)  [fill=magenta] {$6$};
			\node[node]	(7) at (-25mm, 2cm)  [fill=orange] {$4$};				

				\path
				 [dashed] 
				 (2) edge node {} (1) 
				(2) edge node {} (3)
				(3) edge node {} (7)
				(7) edge node {} (4)
				(4) edge node {} (2)
				(4) edge node {} (5)
				(5) edge node {} (6)
				;
		\end{tikzpicture}}
 &  
 $\Delta_3 + \Delta_4$
 \hspace{-0.6in}
{\small \begin{tikzpicture}	[->,>=stealth',shorten >=1pt,auto, node distance=1.5cm, thick,node/.style={circle,draw}]
	                 \node[node]	(1) at (-75mm, 4cm)  [fill=blue!40] {$7$};
			\node[node]	(2) at (-5cm, 4cm)  [fill=green] {$2$};
			\node[node]	(3) at (-5cm, 2cm)  [fill=orange] {$3$};
			\node[node]	(4) at (-25mm, 4cm)  [fill=yellow] {$1$};			
			\node[node]	(5) at (0cm, 4cm)  [fill=green] {$5$};			
			\node[node]	(6) at (0cm, 2cm)  [fill=orange] {$6$};
			\node[node]	(7) at (-25mm, 2cm)  [fill=magenta] {$4$};				

				\path
				 [dashed] 
				 (2) edge node {} (1) 
				(2) edge node {} (3)
				(3) edge node {} (7)
				(7) edge node {} (4)
				(4) edge node {} (2)
				(4) edge node {} (5)
				(5) edge node {} (6)
				;
		\end{tikzpicture}}
&
$\Delta_3 + \Delta_5$
\hspace{-0.6in}
{\small \begin{tikzpicture}	[->,>=stealth',shorten >=1pt,auto, node distance=1.5cm, thick,node/.style={circle,draw}]
	                 \node[node]	(1) at (-75mm, 4cm)  [fill=orange] {$7$};
			\node[node]	(2) at (-5cm, 4cm)  [fill=green] {$2$};
			\node[node]	(3) at (-5cm, 2cm)  [fill=orange] {$3$};
			\node[node]	(4) at (-25mm, 4cm)  [fill=yellow] {$1$};			
			\node[node]	(5) at (0cm, 4cm)  [fill=green] {$5$};			
			\node[node]	(6) at (0cm, 2cm)  [fill=blue!40] {$6$};
			\node[node]	(7) at (-25mm, 2cm)  [fill=magenta] {$4$};						

				\path
				 [dashed] 
				 (2) edge node {} (1) 
				(2) edge node {} (3)
				(3) edge node {} (7)
				(7) edge node {} (4)
				(4) edge node {} (2)
				(4) edge node {} (5)
				(5) edge node {} (6)
				;
		\end{tikzpicture}}		
 \\
\\
\hline
  \\
$\Delta_3 + \Delta_6$
\hspace{-0.6in}
{\small \begin{tikzpicture}	[->,>=stealth',shorten >=1pt,auto, node distance=1.5cm, thick,node/.style={circle,draw}]
	                 \node[node]	(1) at (-75mm, 4cm)  [fill=blue!40] {$7$};
			\node[node]	(2) at (-5cm, 4cm)  [fill=green] {$2$};
			\node[node]	(3) at (-5cm, 2cm)  [fill=orange] {$3$};
			\node[node]	(4) at (-25mm, 4cm)  [fill=yellow] {$1$};			
			\node[node]	(5) at (0cm, 4cm)  [fill=green] {$5$};			
			\node[node]	(6) at (0cm, 2cm)  [fill=blue!40] {$6$};
			\node[node]	(7) at (-25mm, 2cm)  [fill=magenta] {$4$};					

				\path
				 [dashed] 
				 (2) edge node {} (1) 
				(2) edge node {} (3)
				(3) edge node {} (7)
				(7) edge node {} (4)
				(4) edge node {} (2)
				(4) edge node {} (5)
				(5) edge node {} (6)
				;
		\end{tikzpicture}}
 &  
 $\Delta_3 + \Delta_7$
 \hspace{-0.6in}
{\small \begin{tikzpicture}	[->,>=stealth',shorten >=1pt,auto, node distance=1.5cm, thick,node/.style={circle,draw}]
	                 \node[node]	(1) at (-75mm, 4cm)  [fill=orange] {$7$};
			\node[node]	(2) at (-5cm, 4cm)  [fill=green] {$2$};
			\node[node]	(3) at (-5cm, 2cm)  [fill=orange] {$3$};
			\node[node]	(4) at (-25mm, 4cm)  [fill=yellow] {$1$};			
			\node[node]	(5) at (0cm, 4cm)  [fill=brown] {$5$};			
			\node[node]	(6) at (0cm, 2cm)  [fill=blue!40] {$6$};
			\node[node]	(7) at (-25mm, 2cm)  [fill=magenta] {$4$};			

				\path
				 [dashed] 
				 (2) edge node {} (1) 
				(2) edge node {} (3)
				(3) edge node {} (7)
				(7) edge node {} (4)
				(4) edge node {} (2)
				(4) edge node {} (5)
				(5) edge node {} (6)
				;
		\end{tikzpicture}}
		&	
$\Delta_4 + \Delta_5$
\hspace{-0.6in}
{\small \begin{tikzpicture}	[->,>=stealth',shorten >=1pt,auto, node distance=1.5cm, thick,node/.style={circle,draw}]
	                 \node[node]	(1) at (-75mm, 4cm)  [fill=yellow] {$7$};
			\node[node]	(2) at (-5cm, 4cm)  [fill=green] {$2$};
			\node[node]	(3) at (-5cm, 2cm)  [fill=green] {$3$};
			\node[node]	(4) at (-25mm, 4cm)  [fill=green] {$1$};			
			\node[node]	(5) at (0cm, 4cm)  [fill=green] {$5$};			
			\node[node]	(6) at (0cm, 2cm)  [fill=orange] {$6$};
			\node[node]	(7) at (-25mm, 2cm)  [fill=green] {$4$};			

				\path
				 [dashed] 
				 (2) edge node {} (1) 
				(2) edge node {} (3)
				(3) edge node {} (7)
				(7) edge node {} (4)
				(4) edge node {} (2)
				(4) edge node {} (5)
				(5) edge node {} (6)
				;
		\end{tikzpicture}}	
  \\
\\
\hline
  \\
$\Delta_4 + \Delta_7$
\hspace{-0.6in}
{\small \begin{tikzpicture}	[->,>=stealth',shorten >=1pt,auto, node distance=1.5cm, thick,node/.style={circle,draw}]
	                 \node[node]	(1) at (-75mm, 4cm)  [fill=yellow] {$7$};
			\node[node]	(2) at (-5cm, 4cm)  [fill=green] {$2$};
			\node[node]	(3) at (-5cm, 2cm)  [fill=green] {$3$};
			\node[node]	(4) at (-25mm, 4cm)  [fill=green] {$1$};			
			\node[node]	(5) at (0cm, 4cm)  [fill=magenta] {$5$};			
			\node[node]	(6) at (0cm, 2cm)  [fill=orange] {$6$};
			\node[node]	(7) at (-25mm, 2cm)  [fill=green] {$4$};					

				\path
				 [dashed] 
				 (2) edge node {} (1) 
				(2) edge node {} (3)
				(3) edge node {} (7)
				(7) edge node {} (4)
				(4) edge node {} (2)
				(4) edge node {} (5)
				(5) edge node {} (6)
				;
		\end{tikzpicture}}
 &  
 $\Delta_2 + \Delta_4 + \Delta_5$
\hspace{-0.6in}
{\small \begin{tikzpicture}	[->,>=stealth',shorten >=1pt,auto, node distance=1.5cm, thick,node/.style={circle,draw}]
	                 \node[node]	(1) at (-75mm, 4cm)  [fill=yellow] {$7$};
			\node[node]	(2) at (-5cm, 4cm)  [fill=orange] {$2$};
			\node[node]	(3) at (-5cm, 2cm)  [fill=green] {$3$};
			\node[node]	(4) at (-25mm, 4cm)  [fill=green] {$1$};			
			\node[node]	(5) at (0cm, 4cm)  [fill=orange] {$5$};			
			\node[node]	(6) at (0cm, 2cm)  [fill=magenta] {$6$};
			\node[node]	(7) at (-25mm, 2cm)  [fill=orange] {$4$};				

				\path
				 [dashed] 
				 (2) edge node {} (1) 
				(2) edge node {} (3)
				(3) edge node {} (7)
				(7) edge node {} (4)
				(4) edge node {} (2)
				(4) edge node {} (5)
				(5) edge node {} (6)
				;
		\end{tikzpicture}}
		&	
$\Delta_2 + \Delta_4 + \Delta_7$
\hspace{-0.6in}
{\small \begin{tikzpicture}	[->,>=stealth',shorten >=1pt,auto, node distance=1.5cm, thick,node/.style={circle,draw}]
	                 \node[node]	(1) at (-75mm, 4cm)  [fill=yellow] {$7$};
			\node[node]	(2) at (-5cm, 4cm)  [fill=orange] {$2$};
			\node[node]	(3) at (-5cm, 2cm)  [fill=green] {$3$};
			\node[node]	(4) at (-25mm, 4cm)  [fill=green] {$1$};			
			\node[node]	(5) at (0cm, 4cm)  [fill=magenta] {$5$};			
			\node[node]	(6) at (0cm, 2cm)  [fill=blue!40] {$6$};
			\node[node]	(7) at (-25mm, 2cm)  [fill=orange] {$4$};			

				\path
				 [dashed] 
				 (2) edge node {} (1) 
				(2) edge node {} (3)
				(3) edge node {} (7)
				(7) edge node {} (4)
				(4) edge node {} (2)
				(4) edge node {} (5)
				(5) edge node {} (6)
				;
		\end{tikzpicture}}	
  \\
\\
\hline
  \\
$\Delta_3 + \Delta_4 + \Delta_5$
\hspace{-0.6in}
{\small \begin{tikzpicture}	[->,>=stealth',shorten >=1pt,auto, node distance=1.5cm, thick,node/.style={circle,draw}]
	                 \node[node]	(1) at (-75mm, 4cm)  [fill=brown] {$7$};
			\node[node]	(2) at (-5cm, 4cm)  [fill=green] {$2$};
			\node[node]	(3) at (-5cm, 2cm)  [fill=orange] {$3$};
			\node[node]	(4) at (-25mm, 4cm)  [fill=yellow] {$1$};			
			\node[node]	(5) at (0cm, 4cm)  [fill=green] {$5$};			
			\node[node]	(6) at (0cm, 2cm)  [fill=blue!40] {$6$};
			\node[node]	(7) at (-25mm, 2cm)  [fill=magenta] {$4$};					

				\path
				 [dashed] 
				 (2) edge node {} (1) 
				(2) edge node {} (3)
				(3) edge node {} (7)
				(7) edge node {} (4)
				(4) edge node {} (2)
				(4) edge node {} (5)
				(5) edge node {} (6)
				;
		\end{tikzpicture}}
 &  
		&	
  \\
\\
\hline

\end{tabular}
}
\caption{The balanced colourings corresponding to the non join-irreducible elements in the lattice $V_{\N}$ of the synchrony subspaces for the $1$-input regular coupled cell network $\NN$ in Figure~\ref{fig:subN2}.} \label{tab:nonirredu}
\end{center}
\end{table}
\END
\end{ex}

Taking into account the notion of balanced colouring and the results in Lemma~\ref{lem:eigen} and Theorem~\ref{thm:setirredu}, we give next a characterization of the possible patterns of balanced colourings associated to the join-irreducible synchrony subspaces for $\NN$.

\begin{cor}\label{cor:corirred}
Let $\NN$ be a 1-input regular coupled cell network with $n$ cells and a $m$-cell ring, with $1 \le m \le n$. Given a balanced colouring of the cells of $\NN$ associated to a join-irreducible synchrony subspace for $\NN$, we have one of the following:
\begin{itemize}
\item[(i)] all the cells have the same colour,
\item[(ii)] the cells in the ring are coloured with a sequence of $\frac{m}{d}$ different colours, with $d$ a proper divisor of $m$ different from $1$, that is repeated sequentially $d$ times and if we wrap each of the rooted trees over the ring then overlapping cells have the same colour,
\item[(iii)] the cells in the ring have different colours and if we wrap each of the rooted trees over the ring then overlapping cells have the same colour,
\item[(iv)] all the cells with the exception of a nonempty subset of leafs have the same colour and the cells in that subset of leafs have another different colour,
\item[(v)] all the cells with the exception of those in a nonempty subset  $J$ of subtrees, with $s= max \{ depth(\BB)|\ \BB \in J \} > 1$, have the same colour $r$. For each subtree $\BB_i \in J$, cells at different distances from the root cell of the subtree have different colours and cells at the same distance have the same colour. Moreover, given any two subtrees, if a cell in one of the subtrees is distant from the root cell of its subtree the same distance as a cell in the other subtree is distant from the corresponding root cell then the two cells have the same colour.

\end{itemize}
\end{cor}

\begin{ex}  \normalfont 
Consider the balanced colourings in Table~\ref{tab:irredu} that correspond to the  join-irreducible elements of the lattice $V_{\N}$ of the synchrony subspaces for the $1$-input regular coupled cell network $\NN$ in Figure~\ref{fig:subN2} and the classification in Corollary~\ref{cor:corirred}. The colouring $\Delta_1$ satisfies (i), the colouring $\Delta_2$ satisfies (ii), the colouring $\Delta_3$ satisfies (iii), the colourings $\Delta_4$ and $\Delta_5$ satisfy (iv) and the colourings $\Delta_6$,  $\Delta_7$ and $\Delta_8$ satisfy (v).
\END
\end{ex}

Given that each synchrony subspace for a 1-input regular coupled cell network is the sum of join-irreducible synchrony subspaces and considering the possible patterns of balanced colourings associated to the join-irreducible elements in the lattice of synchrony subspaces, described in Corollary~\ref{cor:corirred}, we end up with a charaterization of the possible patterns of balanced colourings for the synchrony subspaces for a 1-input regular coupled cell network.

\begin{cor} \label{cor:corall}
Let $\NN$ be a 1-input regular coupled cell network. Given a balanced colouring of the cells of $\NN$ associated to a synchrony subspace for $\NN$, we have
\begin{itemize}
\item[(i)]  The ring can have all cells with the same colour, the cells coloured in a sequence of colours (with no two consecutive cells with the same colour) that are repeated sequentially, or each cell with a different colour. 
\item[(ii)]  Each tail, when wrapped over the ring can have the colour of each cell coinciding with the colour of the cell of the ring that it overlaps, it may be that this happens only for the first cells of the tail and that each of the subsequent cells has another different colour or it may happen that each cell of the tail has another different colour. 
\item[(iii)]  Given two tails, if we start by wrapping each of them over the ring while the colour of each cell coincides with the colour of the cell of the ring that it overlaps then, fixing any two cells with the same colour, one from each of the tails, the sequence of (different) colours from the first cell of the tail that was not wrapped over the ring and the fixed cell is the same for the two tails. 
\end{itemize}
\end{cor}

\begin{ex}  \normalfont 
Consider, as examples of Corollary~\ref{cor:corall}, the balanced colourings in Tables~\ref{tab:irredu}~and~\ref{tab:nonirredu}. For a more illustrative example of Corollary~\ref{cor:corall} (iii), consider the balanced colouring in Figure~\ref{fig:1-inputNet}.
\END
\end{ex}

\begin{rmk}\normalfont
(i) A 1-input regular coupled cell network where the ring consists of only one cell with a self-loop is a very particular case of a regular Auto-regulation Feed-forward Neural network (AFFNN).  AFFNNs are special types of networks where the cells are arranged in layers, there is no connection among cells in the same layer, the cells in the first layer either receive no input or have self-loops (there is at least one cell with a self-loop), each cell in another layer only receives connections from cells in the previous layer.  A 1-input regular coupled cell network where the ring consists of only one cell  is a AFFNN where the first layer has only one cell and each cell in the other layers receive only one connection from one of the cells in the previous layer. In \cite{ADF17}, Aguiar~\etal  describe the robust balanced colourings that can occur for AFFNNs. Our result in Corollary~\ref{cor:corall} (ii) for this particular type of 1-input networks is in accordance with the results in  \cite{ADF17}, namely Theorem 4.9, where the tails are the paths.
(ii) If $\NN$ is a 1-input regular coupled cell network and we identify all the cells in the ring we get a network, called a quotient network, that is a 1-input network as in (i). 
Each balanced colouring for the quotient network lifts to a balanced colouring for the network $\NN$ where  the cells in the rooted trees keep the same colours and all the cells in the ring have the colour of the lifted cell, the cell with the self-loop. Again, our result in Corollary~\ref{cor:corall} (ii) for this particular type of balanced colourings of 1-input networks is in accordance with the results in  \cite{ADF17}.
\END
\end{rmk}

\section{Conclusions} \label{sec:conclusions}
Our main results are in Section~\ref{sec:lattice_1_regular}. Corollary~\ref{lem:eigen} allows us to identify the eigenvectors of the adjacency matrix of a 1-input regular network without having to compute them and to obtain the Jordan chains directly from the subtrees of the directed rooted trees of the network graph. Consequently, we are able to identify, in Theorem~\ref {thm:setirredu}, the set of join-irreducible synchrony subspaces for the lattice of synchrony subspaces of a 1-input regular coupled cell network. By Aguiar~\etal~\cite{AADF11}, this set is join-dense, that is, we are able to get the all lattice through the join (sum) of these join-irreducible elements. 

Synchrony-breaking bifurcations for 1-input regular networks can be relevant from the point of view of applications. For example, in \cite{GPSZ09}, Golubitsky~\etal\  discuss how the periodic forcing of the first node in a chain of coupled identical systems, corresponding to a 1-input network with three-cells, whose internal dynamics is each tuned near a point of Hopf bifurcation, can lead naturally to successive amplification of the incoming signal. Their results have implication for certain models of the auditory system, in particular, models of the basilar membrane and attached hair bundles. 
The identification of the elements in the lattice of synchrony subspaces of a network is important in the study of synchrony-breaking bifurcations. Given a 1-input regular network it is straightforward to identify, from our results, for example, the three-dimensional synchrony subspaces for that network. Given a codimension one steady-state or Hopf bifurcation from a synchronous equilibrium for that network, using then the results in Leite and Golubitsky~\cite{LG06}, on the classification of the codimension one steady-state and Hopf bifurcations from a synchronous equilibrium in three-cell homogeneous, we are able to identify the new bifurcating branches with three-dimensional synchrony pattern, together with their stability. 

Combining the results in Section~\ref{sec:lattice_1_regular} with those in Section~\ref{sec:lattice_union}, on how to obtain the lattice of synchrony subspaces for a disjoint union of networks from the lattices of synchrony subspaces of the component networks, we get the lattice of synchrony subspaces for the edge-type subnetworks $\NN_{\E_i}$ of a homogeneous network $\NN$ with asymmetric inputs. By Aguiar and Dias~\cite{AD14}, the lattice of synchrony subspaces for $\NN$ is given by the intersection of the lattices of synchrony subspaces for the edge-type subnetworks $\NN_{\E_i}$. 
The results in Section~\ref{sec:lattice_1_regular} give then a more expeditious and efficient way of calculating the lattice of synchrony subspaces for the particular case of homogeneous (regular) coupled cell networks with asymmetric inputs, rather than the one presented in Aguiar and Dias~\cite{AD14} for general homogeneous (regular) coupled cell networks. Moreover, they allow the implementation of an efficient algorithm without problems of numerical rounding.

The results presented here extend trivially to a particular type of weighted networks with identical cells, where edges of the same type have associated the same weight or strength and different edge-types correspond to different weights. 
In particular, it is immediate to conclude that analogous results to those in Lemma~\ref{lem:eigen}, and thus in Theorem~\ref{thm:setirredu}, also occur for adjacency matrices of the form (\ref{eq:matrix} ) such that the nonzero entries are all equal to a value $k \in \RR \setminus \{1\}$, as happens with the adjacency matrices of such networks. An application of this type of networks to animal locomotion appears in Golubitsky~\etal~\cite{GSBC98},~\cite{GSBC99}, Buono and Golubitsky~\cite{BG01}  and  Stewart~\cite{S17} and has obvious utility in robotics (see, for example, Righetti and Ijspeert~\cite{RI06} and  In and Palacios~\cite{IP17}). In \cite{GSBC98}, ~\cite{GSBC99}, Golubitsky~\etal \  propose a class of such particular kind of weighted networks which provide models for the gaits of $2n$-legged animals. The analysis made there for an eight-cell network modelling the quadruped locomotion was later extended by Buono and Golubitsky~\cite{BG01}  and very recently by Stewart~\cite{S17}.

 \subsection*{Acknowledgments} 

The author thanks Pedro Soares for useful comments.

\vspace{.1in}


\end{document}